%% file: Poincare26.tex
\documentclass[a4paper,10pt]{article}
\usepackage{geometry,times}
\usepackage{amsfonts}
\usepackage{amssymb}
\usepackage{amsmath}
\usepackage{tikz}
\usepackage{pgfplots}
\usepackage{tikz-3dplot}
\tdplotsetmaincoords{60}{115}
\pgfplotsset{compat=newest}
\definecolor{paleyellow}{rgb}{1, 1, 0.9}
\usepackage{amsthm}
\newcommand{\rose}{R\accent 23 u\v zi\v cka}
\newcommand{\bb}{\mathbb}

\theoremstyle{plain}
\newtheorem{satz}{prop}
\newtheorem{corollary}[satz]{Corollary}
\newtheorem{lem}[satz]{Lemma}
\newtheorem{definition}[satz]{Definition}
\newtheorem{prop}[satz]{Proposition}
\theoremstyle{thm}
\newtheorem{theorem}[satz]{Theorem}
\theoremstyle{note}
\newtheorem{note}[satz]{Note}
\newtheorem{remark}[satz]{Remark}
\theoremstyle{problem}

\newtheorem{example}[satz]{Example}
\theoremstyle{notadef}
\newtheorem{notation}[satz]{Notation}
\usepackage[]{graphicx}
\usepackage{color}
\input{color_pdf.tex}
 \usepackage{titlesec}
\titleformat{\section}
{\normalfont%\sffamily
  \bfseries%\color{cyan}
}
{\thesection}{1ex}{}
\titleformat{\subsection}
{\normalfont%\sffamily
  \bfseries%\color{cyan}
}
    {\thesubsection}{1ex}{}
\newcommand{\NN}{\ensuremath{\mathbb{N}}}

\newcommand{\bz}{\underline{{z}}}
\newcommand{\cA}{{\mathcal{A}}}

\newcommand{\bx}{\underline{x}}

\begin{document}
%%%%%%%%%%%%%%%%%%
\hspace*{1cm}\\[.4cm]
{\large \textbf{Exact Poincaré Constants in
  $n$-dimensional Annuli}}
\\[.5cm]
Bernd Rummler\\
\hspace*{2mm}{\small \em Otto-von-Guericke-Universit\"at Magdeburg,
  Inst. f\"ur Analysis und Numerik, PF
  4120, 39016 Magdeburg}
\\[2mm]
Michael \rose\\
\hspace*{2mm}{\small \em Inst. f{\"u}r Angewandte Mathematik, Universit{\"a}t Freiburg,
  Ernst-Zermelo-Str.~1, 79104 Freiburg}
\\[2mm]
Gudrun Th\"ater\footnote{Corresponding author\quad
  E-mail:~\textsf{gudrun.thaeter@kit.edu}}\\
\hspace*{2mm}{\small \em Inst. f\"ur Angewandte \& Numerische Mathematik, KIT, 
  76128 Karlsruhe}
\\[5mm]
%%%%%%%%%%%%%%%%%
\hspace*{.1cm}\hfill\parbox{16cm}{
 \\[2mm]
{\small\textbf{Abstract:}
 We study $n$-dimensional annuli for $n\,\in\,\{2,\dots,N\}$ with
 $N\,<\,\infty$. 
 We choose a non-dimensional setting such that for any fixed $n%\,\in\,\{2,..,N\}
 $ and given  number ${\cal A}>0$ the 
 annuli ${\Omega}_{(n),\cal A}$ are defined
as space between two concentrical 
balls with radii ${\cal A}/2$ and ${\cal A}/2 +1$ in ${\bb
  R}^{n}$. For these
geometries we provide calculated (precise) Poincar\'e constants. 
These depend on ${\cal A}$ and the dimension $n$. 
Additionally we find a direct match of the Poincar\'e constants
for solenoidal vector fields in ${\bb R}^{n}$ and the Poincar\'e constants
for scalar functions in ${\bb
  R}^{n+2}$ (all with vanishing Dirichlet traces). 
This is based on the relation of the 
first eigenvalues and one eigenfunction
of the (scalar) Laplace and the Stokes operator. 
In addition we consider the limit ${\cal
  A}\,\to\,0$.  In this context problems in domains
${\Omega}_{(n),\sigma}^{*}$ (cf. \cite{RumTh2024})
are investigated.
These domains
enable us to use the Green's function of the Laplacian 
with vanishing 
Dirichlet traces 
to show that 
the first eigenvalue here 
tends to the first eigenvalue of the corresponding problem on the open
unit ball in ${\bb
  R}^{n}$. 
On the other hand, we take advantage of the so-called small-gap limit
for  ${\cal A}\to\infty$.% like in our previous papers to Poincar\'e constants in annuli. 
}}
%%%%%%%%%%
\newcommand{\abs}[1]{\left | #1 \right |}
\newcommand{\RR}{\mathbb{R}}
\newcommand{\cC}{{\mathcal{C}}}
%\newcommand{\Vi}{\varphi}
%%%%%%%%%%%%%%%%%%%%%%%%%%%%%%%%%%%%%%%%%
\section{Introduction\label{sec_int}}
Flow in annulus domains in 
dimensions  $n=2$ or $n=3$ is a model for a lot of
applications. Some examples are natural convection 
in the space between concentric pipes ($n=2$)
and the modelling of astrophysical resp. geophysical
applications ($n=3$). Due to the different possible thickness of the flow
domain these relate to very diverse 
objects. They have in common that one has to understand 
fluid flow in spherical shells - i.e. flow in the domain between
two concentric spheres (in space dimensions $2$ or $3$). 
Fig.~\ref{fig:geometry} shows the geometry with general inner and
outer radii $R_i$ and $R_o$, respectively,  which is normalized later
on using the non-dimensional quantities defined in (\ref{calA}) below. In this
context we published a series of papers (see \cite{RuRuTh2025} and the
references therein) and  could not avoid
to compare
differences and similarities between the cases $n=2$ and $n=3$. In particular, in \cite{RuRuTh2016} 
we calculate the Poincar\'e
constants as function of 
a non-dimensional number, which characterises the (relative) size of
the gap (cf. also \cite{RuRuTh2025}). We
 became curious how our techniques translate and apply to higher
dimensions and if there is a pattern for the Poincar\'e constants
in these higher dimensions. Now, 
this is the focus of our present paper. In particular we first provide the 
mathematical characteristics of the  
Poincar\'e constants for scalar functions in ${\bb R}^{n}$, $n >
1$. Since the underlying application is fluid flow we are also interested in the
special case that the objects are solenoidal vectors. 
Here we will see that one of the first Stokes eigenfunctions fixes the related 
Poincar\'e constant. We say ``one of'' since, e.g. the first eigenvalue in
dimension $n=3$ has multiplicity $3$.
For this we formulate a shift theorem between the Poincar\'e constants
for vectorial functions and  
scalar functions.% to benefit from our new results for scalar functions in ${\bb R}^{n}$.

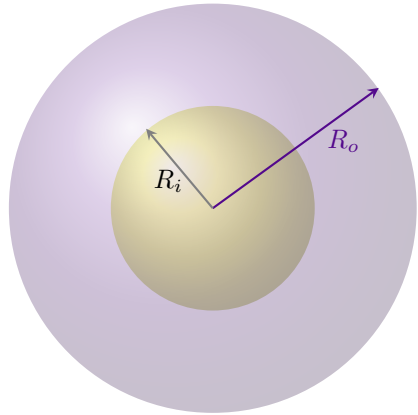
\begin{figure}[th]
\centering
\begin{tikzpicture}[scale=1.8,tdplot_main_coords] 
% Create a point (P)
\coordinate (P) at ({0},{1.905/sqrt(2)},{1.905/sqrt(2)
});
\coordinate (P1) at ({0},{-0.765/sqrt(2)},{0.765/sqrt(2)});
% Draw shaded circle
\shade[ball color = vivid-viol,
    opacity = 0.25
] (0,0,0) circle (1.5cm);the the 
\shade[ball color = yellow,
    opacity = 0.3
] (0,0,0) circle (.75cm);
 
% draw arcs 
 
% Projection of the point on X and y axes
%\draw[thin, dashed] (P) --++ (0,0,{-1/sqrt(3)});
%\draw[thin, dashed] ({1/sqrt(3)},{1/sqrt(3)},0) --++
%(0,{-1/sqrt(3)},0);
%\draw[thin, dashed] ({1/sqrt(3)},{1/sqrt(3)},0) --++
%({-1/sqrt(3)},0,0);
 
% Axes in 3 d coordinate system
%\draw[-stealth] (0,0,0) -- (1.70,0,0);
%\draw[-stealth] (0,0,0) -- (0,1.30,0);
%\draw[-stealth] (0,0,0) -- (0,0,1.30);
%\draw[dashed, gray] (0,0,0) -- (-1,0,0);
%\draw[dashed, gray] (0,0,0) -- (0,-1,0);
 Call Center - Kanyshay KAYA <callcenter4@corneliaresort.com>
% Line from the origin to (P)
\draw[thick,vivid-viol, -stealth] (0,0,0) -- (P);
\draw(0.5,1.5,1.2) node[anchor=east,vivid-viol]{$R_o$};
% Add small circle at (P)
\draw[thick,gray, -stealth] (0,0,0) -- (P1);
\draw((0.25,-0.05,0.35) node[anchor=east]{$R_i$};
%\draw[fill = lightgray!50] (P) circle (0.5pt);
% \draw[->, >=stealth] (0,0)--(-0.4,1.96);
%   \draw[->, >=stealth] (0,0)--(-0.7,0.714);
%\draw(-0.3,0.3) node[anchor=east]{$R_i$}; 
%
\end{tikzpicture}
\caption{$2$d Scetch of $n$-dimensional annulus  with inner/outer radius%  in
  % a Cartesian coordinate system
}
\label{fig:geometry}
\end{figure}
\noindent To have a non-dimensional setting we use the following two well established frames: Either the so-called {\em
  inverse relative gap width} (mathematically more precise would be %to
                                %call it the
{\em gap weighted inner diameter}) $\cA$ or the parameter ${\sigma}$.
The first is defined as
\begin{equation}\label{calA}
{\cal A}\,:=\,\frac{\displaystyle{2 R_{i}}}{\displaystyle{ R_{o}
    -R_{i}}}
\qquad\mbox{while}\qquad
{\sigma}\,:=\,\frac{\displaystyle{R_{i}}}{\displaystyle{ R_{o}}}=
\frac{\displaystyle{\cal A}}{\displaystyle{ {\cal A} +2 }}\,\,,
\end{equation}
for $R_{i}$ and $R_{o}$  denoting the inner and
outer radius of the annulus, respectively, and $0\,<\,R_{i}<\,R_{o}$
(see Fig.~\ref{fig:geometry}).\\[1mm]
Our \textbf{main results} are: We provide exact 	
conditional equations together with numerical approximations and exact
values for the Poincar\'e constants 
for scalar fields on spherical shells in ${\bb R}^{n}$, $n\,\in\,{\bb
  N}, \,1< n <\infty$ for any  
${\cal A}\,\in\,[0,\infty)$ 
as
well as the first eigenvalues of the Laplace operators in these configurations
%on spherical shells in ${\bb R}^{n}$, $n\,\in\,{\bb N}, \,2 \leq n
%<\infty$ for ${\cal A}\,\in\,[0,\infty)$.
(all {\em including the limits}, i.e. ${\cal A}$ infinitely large and
zero).
In addition we find the Poincar\'e constants for solenoidal vector
fields and the first eigenvalues for the Stokes
operator in these configurations including the limiting cases.
\\[1mm]
{\bf General notation A.} Let ${\bb R}^{n}$ be endowed with the
usual Euclidian norm $\| . \|$. 
Elements of ${\bb R}^{n}$ are denoted by underlined small letters.
We write ${\Omega}_{(n)}^{*}:=\{{\underline{x}} \in {\bb R}^{n}:\,
\|{\underline{x}}\|\,
<1\}$ for the open unit ball and use $\omega_{(n)}\,:=\,\{{\underline{x}}\in {\bb
  R}^{n}:\,\|{\underline{x}}\|=1\}$ for its surface.
For $r\,\in\,(0,\infty)$ the closed spherical surfaces around the origin with radius $r$ are
$\omega_{(n),r}\,:=\,\{{\underline{x}}\in {\bb
  R}^{n}:\,\|{\underline{x}}\|=r\} $
and the surface areas of $\omega_{(n)}$ are 
$|\omega_{(n)}|\,=\,{2\pi}^{\frac{n}{2}}/
{\Gamma(\frac{n}{2})}$ ($\forall  n\,\in\,{\bb N}$), where ${\Gamma(.)}$ is the ${\Gamma}$-function.
\\[1mm]
{\bf Annulus domains in ${\bb R}^{n}$}. %We study our domains without
                                %scalar dimensions (excluding the
                                %dimension of ${\bb R}^{n}$).
%As usual we pick the annuli (the sherical shells) with fixed gap
%width $1$ using our non-dimensional parameter ${\cal A}$.
%%
For any ${\cal A}\,\in\,(0,\infty)$, we denote the corresponding annulus (the sherical shell) 
by
\[
{\Omega}_{_{(n)},{\cal A}}\,:=\{{\underline{x}} \in {\bb R}^{n}:\,{\cal A}/{2}
<\|{\underline{x}}\|<1+{\cal A}/{2}\}\,.
\]
Its boundary 
$\partial{\Omega}_{_{(n)},
\cal A}$ consist of the two parts, namely, 
$\omega_{{(n)},{\cal A}/{2}}$ and  $\omega_{(n),{1+{\cal A}/{2}}}\,.$  Moreover,
 for all   
${\sigma}\,\in\,(0,1)$ we introduce the family of annuli \[
{\Omega}^{*}_{(n),\sigma}:=\{{\underline{x}} \in {\bb
  R}^{n}:\,0<\sigma<\|{\underline{x}}\|\,<1\}\,.
\]
In this notation we can directly use our results from \cite{RumTh2024}
which is more convenient for the case ${\cal A}\,\to\,0$.
The advantage is that these $n$-dimensional annuli are subsets of the
unit ball 
${\Omega}^{*}_{(n)}$. They have  the boundary
$\partial{\Omega}^{*}_{(n),\sigma}=
\omega_{(n),\sigma} \cup \omega_{(n)}$. 
\\[1mm]
%%%%%%%%%%%%%%%%%%%%%%%%
{\bf General notation B.} Let ${\Omega}$ stand as shorthand for any of the domains
defined above %${\Omega}_{{\cal A}},\, {\Omega}^{*}_{\sigma},$ and ${\Omega}^{*}$
and the abbreviation $(.)$ for $({\Omega})$, respectively. 
We consider the usual Lebesgue and Sobolev spaces ${\bb L}_{2}(.)$ and
${\bb W}_{2}^{k}(.)$ 
of scalar functions and %the Lebesgue and Sobolev spaces 
${\underline{\bb L}}_{{\;\!}{2}}(.)=({\bb L}_{2}(.))^{n}$
and ${\underline{\bb W}}_{{\;\!}{2}}^{k}(.)=({\bb
  W}_{2}^{k}(.))^{n}$ of vector functions. 
The norm in ${\bb L}_{2}(.)$ is denoted by $\| . \|_{2}$,
${\bb W}_{2}^{1}\hspace{-.62cm}{~}^{{~}^{{~}^{o}}}\hspace{.2cm}(.)$ is
the closure of $C_{o}^{\infty}(.)$ 
in ${\bb W}_{2}^{1}(.)$.  
All solenoidal vector functions belonging to 
${\underline{C}}_{{\;\!}{o}}^{\infty}(.)$ form $\underline{\cal V}(.)$. The closures
of $\underline{\cal V}(.)$ in ${\underline{\bb L}}_{{\;\!}{2}}(.)$ 
and
${\underline{\bb W}}_{{\;\!}{2}}^{1}(.)$, respectively, are denoted by
${\underline{\bb H}}(.)$ 
and ${\underline{\bb V}}(.)$, respectively.  We use the spherical
Bessel functions ${J_{k}(.)}$ of the first kind as well as 
the spherical Bessel functions  ${J_{-k}(.)}$ 
of order  $k\,\in\, \{\frac{1}{2}+m,\,m\,\in\, {\bb N}_{o}\} $ resp. Weber's functions  ${Y_{k}(.)}$ of the order  
$k\,\in\, {\bb N}_{o}$. The  Weber functions ${Y_{k}(.)}$ of the order  
$k\,\in\, {\bb N}_{o}$ are also called 
Bessel functions of the second kind or sometimes
Neumann functions (e.g., \cite{AAR}, \cite{CouHil}, \cite{Triebel}
etc.).
%%%%%%%%%%%%%%%%%%
\begin{notation}\label{N3} In ${\bb
  R}^{n},\,n\geq 3$, let the unit vectors in the Cartesian coordinate
system %in ${\bb R}^{n},\,n\geq 3$
be given by ${\underline{\mathfrak{e}}}_{j}\,:=\,
(\delta_{j,1},\delta_{j,2}, \dots, \delta_{j,n})^{T}$  ($\forall\,j=1,2,\dots,n$, with
Kronecker's delta $\delta_{j,k}$). 
The polar coordinates are  $r$,  $\vartheta_{1}$, $\dots$ , $\vartheta_{n-2}$ and $\varphi$ with the corresponding unit vectors
${\underline{\mathfrak{e}}}_{{\;\!}r}$, ${\underline{\mathfrak{e}}}_{{\;\!}\vartheta_{1}}$, $\dots$, ${\underline{\mathfrak{e}}}_{{\;\!}\vartheta_{n-2}}$  and
${\underline{\mathfrak{e}}}_{{\;\!}\varphi}$.
If we denote by $\{{\underline{\mathfrak{e}}}_{r},
{\underline{\mathfrak{e}}}_{\vartheta_{1}}, \dots ,
{\underline{\mathfrak{e}}}_{\vartheta_{n-2}},{\underline{\mathfrak{e}}}_{\varphi}\}$
the system of these unit vectors in spherical polar coordinates then ${\underline{u}}$
is representable in both systems as\\[.1cm]
\hspace*{1.3cm}${\underline{u}}\,=\,\sum_{j=1}^{n} u_{j}{\underline{\mathfrak{e}}}_{j}\,=\, \sum_{j=1}^{n} u_{j,{\mathfrak{c}}}{\underline{\mathfrak{e}}}_{j}\,=\, 
 u_{r}{\underline{\mathfrak{e}}}_{r}+
 \sum_{k=1}^{n-2}
 u_{\vartheta_{k}}{\underline{\mathfrak{e}}}_{\vartheta_{k}}+
 u_{\varphi}{\underline{\mathfrak{e}}}_{\varphi}\,=\, 
 u_{r,{\mathfrak{s}}}{\underline{\mathfrak{e}}}_{r}+
\sum_{k=1}^{n-2}
 u_{\vartheta_{k},{\mathfrak{s}}}{\underline{\mathfrak{e}}}_{\vartheta_{k}}
+
 u_{\varphi,{\mathfrak{s}}}{\underline{\mathfrak{e}}}_{\varphi}$.\\[.1cm] 
The transformation from one coordinate system to the other is %can be written as 
${\underline{u}}_{\mathfrak{c}}\,=\,
{\underline{\underline{T}}}_{{\mathfrak{c}},{\mathfrak{s}}}{\underline{u}}_{\mathfrak{s}}$
or ${\underline{u}}_{\mathfrak{s}}\,=\,
{\underline{\underline{T}}}^{-1}_{{\mathfrak{c}},{\mathfrak{s}}}
{\underline{u}}_{\mathfrak{c}}\,=\,
{\underline{\underline{T}}}_{{\mathfrak{s}},{\mathfrak{c}}}{\underline{u}}_{\mathfrak{c}}$, respectively
% where we have used the concept of
(these use columns of coordinates).
The transformation matrices ${\underline{\underline{T}}}_{{\mathfrak{c}},{\mathfrak{s}}}$ and ${\underline{\underline{T}}}_{{\mathfrak{s}},{\mathfrak{c}}}$ are given in the Appendix. 
\end{notation}
\noindent The {\bf Poincar\'e-(Friedrichs-)inequalities} are the central tools
to ensure, that  the spaces ${\bb
  W}_{2}^{1}\hspace{-.62cm}{~}^{{~}^{{~}^{o}}}\hspace{.2cm}(.)$ and  
${\underline{\bb V}}(.)$ %additionally
can be equipped with equivalent
norms generated by the Dirichlet norms 
%(which otherwise would only be semi-norms):
\begin{align} \label{Dirichlet} 
\|u\|_{D}  :=
\Big(\sum_{k=1}^{n}\Big\|{\frac{\displaystyle{\partial
  u}}{\displaystyle{\partial x_{k}}}}\Big\|_{2}^{2} 
\Big)^{1/2}\quad \forall \, u\,\in\,{\bb
  W}_{2}^{1}\hspace{-.62cm}{~}^{{~}^{{~}^{o}}}\hspace{.2cm}(.)\,,\quad 
\|{\underline{u}}\|_{D,S}   := 
\Big(\sum_{j,k=1}^{n}
\Big\|
{\frac{\displaystyle{\partial u_{j}}}{\displaystyle{\partial x_{k}}}}
\Big\|_{2}^{2}
\Big)^{1/2}
\quad \forall \,  {\underline{u}}  \,\in\,{\underline{\bb V}}(.)\,,
\end{align}
where the so-called Frobenius inner product is part of the last
definition. 
Denoting by $c_{p}({\cal A})$ and $c_{p,S}({\cal A})$ the Poincar\'e
constants with respect to the spaces ${\bb
  W}_{2}^{1}\hspace{-.62cm}{~}^{{~}^{{~}^{o}}}\hspace{.2cm}({\Omega}_{{\cal
    A}})$ and ${\underline{\bb V}}({\Omega}_{{\cal A}})$,
respectively, the Poincar\'e-(Friedrichs-)inequalities are 
\begin{align}\label{Poinc}
\|u\|_{2}  \leq   c_{p}({\cal A})   \|u\|_{D} \quad\forall \, u\,
\in\,{\bb W}_{2}^{1}\hspace{-.62cm}{~}^{{~}^{{~}^{o}}}
\hspace{.2cm}({\Omega}_{\cal A})\,\quad
\mbox{and}\quad
\|{\underline{u}}\|_{{\underline{\bb L}}_{{\;\!}{2}}}
\leq c_{p,S}({\cal A})
\|{\underline{u}}\|_{D,S} \quad
\forall \,  {\underline{u}}  \,\in\,
{\underline{\bb V}}({\Omega}_{\cal A})\,.
\end{align}
%%%%
The Poincar\'e constants 
are related to the 
first eigenvalue of the Laplace or Stokes operator on the
${\Omega}_{\cal A}$-domains 
(with vanishing Dirichlet traces), respectively,
because of the relations (see the Ths. in Subsections 4.5.3 and 4.5.4 and Th. 3 in
6.1.5 in \cite{Triebel})
\begin{equation} \label{simple-tool}
c_{p}({\cal A}) = 
\big(\lambda_{1,L}({\cal A})\big)^{-1/2}
\quad
\,\quad \mbox{and}\quad
c_{p,S}({\cal A}) =
\big(\lambda_{1,S}({\cal A}) \big)^{-1/2} \,,
\end{equation}
where ${\lambda}_{1,L}({\cal A})$ and  ${\lambda}_{1,S}({\cal A})$
denote the first simple eigenvalue of the Laplace operator and the
first eigenvalue of the Stokes  
operator, respectively.  
We refer to \cite{RumTh2024} for details 
with respect to the eigenvalues and eigenfunctions of the Stokes
operator on ${\Omega}^{*}_{(3)}$ resp. ${\Omega}^{*}_{(3),\sigma}$ and
to \cite[Subsection 6.4.4]{Triebel} for the eigenvalues and
eigenfunctions of the Laplacian on %the open unit ball
${\Omega}^{*}_{(3)}$.  
\\[.3cm]
%%%%
Our paper is organised as follows:
We collect our theoretical fundament in
Section \ref{Sec2}. There we sketch 
the procedures to construct the Laplace as well as the Stokes operator
as Friedrichs' extension from the Poisson and the Stokes
problem, respectively. 
We benefit from the properties of operators with a pure real
point spectrum. 
We introduce the {\em Leray-Helmholtz projector} 
$\Upsilon : {\underline{\bb
    L}}_{{\;\!}2}(.)\,\longmapsto \,{\underline{\bb H}}(.)$ %as the projector onto solenoidal vector fields, 
as well as the criteria for the smallest eigenvalues
${\lambda}_{1,L}({\cal A})$, ${\lambda}_{1,S}({\cal A})$, 
and their corresponding eigenfunctions.
Section \ref{Inv_Lim}  is devoted to the limiting cases. We carefully
conduct the transition ${\cal A}\,\to\,0$ especially for 
${\lambda}_{1,L}({\cal A}\,\to\,0)$ in the form of
${\lambda}_{1,L}({\sigma}\,\to\,0)$. 
The investigations are performed with the Green's functions for
circular annuli ${\Omega}^{*}_{(n),\sigma}$ 
and for the unit ball ${\Omega}_{(n)}^{*}$.
The crucial result here is,
that as $\sigma\,\to\,0$ the  {\it{problem forgets}} the 
center point together with the boundary condition there.

The study of the behaviour for ${\lambda}_{1,S}({\cal A}\,\to\,0)$ is
much easier, because %of the vanishing of
the first eigenfunction of the Stokes operator for
${\Omega}_{(n)}^{*}$ vanishes in  
${\underline{x}}={\underline{0}}$.
Finally, by simple transformations we show, that the cases
${\lambda}_{1,L}({\cal A}\,\to\,\infty)$ 
and ${\lambda}_{1,S}({\cal A}\,\to\,\infty)$ are covered by the
so-called {\em small gap limit.} In Section \ref{Sec4} the values of
the Poincar\'e constants   
$c_{p}({\cal A})\,=\,{\frac{1}{\pi}}$ 
and the calculated values for the Poincar\'e constants $c_{p,S}({\cal A})$ are 
represented as a graph  for
${\cal A}\in [0,\infty)$. 
%%%%%%%%%%%
\section{Theoretical groundwork}\label{Sec2}
\subsection{Available Bounds for the Poincaré constant}\label{Sec2avBou}
%%%%%%%%%
In \cite{RuRuTh2016} and \cite{RuRuTh2025}	we collected available
rules of thumbs for bounds for the Poincaré constant  $c_p$ on domains
$\Omega\subset \bb 
R^3$. A typical example (see, e.g., \cite[Th.~4.1]{galdi1998}) is 
$c_p \leq {\text{diam}(\Omega)}/2$, i.e., it depends on the diameter
of the considered domain $\Omega$. Another bound for
${\Omega}_{(3),\cal A}$ is found applying a result from
\cite{nazarov2000}, namely, 
\begin{equation*}
\label{eq:bound_new_3}
  c_p \leq {\frac{R_o}{R_i}} \, \frac{1}{\pi}
=\frac{1}{\pi} \, \left( {1 + \frac{2}{\mathcal A}}\right)\,.
\end{equation*}
One can extrapolate this for any $n\,\in\,{\bb N},\, n\,\geq \,3$, 
${\Omega}_{(n),\cal A}\,\subset\,\bb R^n $ and ${\cal A}\in (0,\infty)$ as:
\begin{equation}
\label{eq:bound_new_n}
  c_p \leq \,\frac{1}{\pi} \left({\frac{R_o}{R_i}}\right)^{\frac{n-1}{2}} 
=\,\,\frac{1}{\pi} \, \left( {1 + \frac{2}{\mathcal A}}\right)^{\frac{n-1}{2}}\,.
\end{equation}
The proof is a simple repetition of our proof
in \cite{RuRuTh2025}. 
The point is, that  ${\lambda}_{1,L}({\cal A})=(c_p({\cal A}))^{-{2}}$ and that the Poincaré constant 
attains its maximal value on a radial function
$\tilde w({r}) = \tilde w(\|{\underline{x}}\|)=
w({\underline{x}})$  $\in\, {\bb
  W}_{2}^{1}\hspace{-.62cm}{~}^{{~}^{{~}^{o}}}\hspace{.2cm}({\Omega}_{(n),\cal
  A})$ :
\begin{align}
\label{eq:poincarepolar}
\hspace*{-1.62cm}
   c_p^2 \,=\, \max_w \frac{\int_{{\Omega}_{(n), \cal A}} \left|w\right|^2 \, 
d\underline{x}}{\int_{{\Omega}_{(n), \cal A}} ({\underline{\nabla}} w)^{T} \cdot {\underline{\nabla}} w \, d\underline{x}}
= \max_{\tilde{w}} \frac{\int_{{\Omega}_{(n), \cal A}} \!\!  r^{n-1} \left|\tilde w\right|^2 \, d \omega_{(n)}
  d {r}
  }
     {\int_{{\Omega}_{(n), \cal A}} \!\! r^{n-1} ({\underline{\nabla}} {\tilde{w}})^{T} \cdot {\underline{\nabla}} {\tilde{w}} \,d \omega_{(n)}
  d {r}
  }\,. 
\end{align}
Furthermore, for the one-dimensional
Laplace eigenvalue problem we know (e.g. from \cite{nazarov2000} )
that 
$\tilde w$ considered as a function of $r$ belongs to $ {\bb
  W}_{2}^{1}\hspace{-.62cm}{~}^{{~}^{{~}^{o}}}\hspace{.2cm}(\frac{\mathcal
  A}{2},1+\frac{\mathcal A}{2})$.  In particular, we conclude
\begin{equation}
\label{eq:poincareEins}
\left(\frac{1}{\pi}\right)^2 = 
\max_{\tilde w } 
\frac{
\int_{\frac{\mathcal A}{2}}^{1+\frac{\mathcal A}{2}}
\abs{\tilde w(r)}^2 \, dr}
{
\int_{\frac{\mathcal A}{2}}^{1+\frac{\mathcal A}{2}}
\abs{\tilde w'(r)}^2 \, dr}\,\,\,.
\end{equation}
Equation \eqref{eq:poincareEins} is significant in the proof of \eqref{eq:bound_new_n} and for large
$\mathcal A$ (i.e. the small gap limit ${\cal A}\,\to\,\infty$).
\subsection{Laplace and Stokes operators on  n-dimensional balls and  n-dimensional annuli} \label{sec_theo} 
%%%%%%%%%%%%%%%%%%%%%%%
In the following we take both symbols ${\Omega}$ and $(\cdot)$ as
placeholders (as previously in General notation B).
\begin{definition}\label{D4} {\em The Laplace operator is
  defined as 
\begin{align*}
{\boldsymbol L^{\circledast}}\,{v} := -\Big(
{\displaystyle{\frac{{\partial}^{2} v}{\partial x_{1}^{2}}}}+{\displaystyle{\frac{{\partial}^{2} v}{\partial x_{2}^{2}}}}+\dots\,
{\displaystyle{\frac{{\partial}^{2} v}{\partial x_{n}^{2}}}}\Big)
=- \Delta{\;\!}_{\underline{x} } v
\hspace*{0.7cm}  \forall\,v\,\in\,D({\boldsymbol
  L^{\circledast}})=C_{o}^{\infty}({\Omega})\,. 
\end{align*}
We denote Friedrichs' extension of ${\boldsymbol L^{\circledast}}$
by ${\boldsymbol 
  L}:={\overline{\boldsymbol L^{\circledast}}}$, where ${\boldsymbol
  L}$ is defined on 
$D({\boldsymbol L})\,:=\,{\bb
  W}_{2}^{1}\hspace{-.62cm}{~}^{{~}^{{~}^{o}}}\hspace{.2cm}({\Omega}) 
\cup {\bb W}_{2}^{2}({\Omega})$.}
\end{definition}
\begin{remark} The range of ${\boldsymbol L}$ is
$R({\boldsymbol L})={\bb L}_{2}({\Omega})$. In this sense we may write:
${\boldsymbol L}=-\Delta{\;\!}_{\underline{x} } : D({\boldsymbol
  L})\,\longmapsto \,{\bb L}_{2}(.)$.
\end{remark}
%%%%%%
\noindent We need the Leray-Helmholtz projection $\Upsilon$ to define the Stokes operator. 
$\Upsilon$ is the well-defined
projector of ${\underline{\bb L}}_{{\;\!}2}(.)$ onto its subspace 
${\underline{\bb H}}(.)$
of generalised solenoidal fields with vanishing generalised traces in
the normal direction. %of the boundary. 
We note, that %the Leray-Helmholtz projector $\Upsilon$
it is also used in the sense of:
$\Upsilon :{\underline{\bb W}}_{{\;\!}2}^{1}(.)\,\longmapsto
\,{\underline{\bb V}}(.)$\,.
\begin{definition}\label{D5} 
{\em The Stokes operator is defined as
$
{\boldsymbol S^{\circledast}}\,{\underline{v}}:=-
\Delta_{\underline{x} } {\underline{v}}\hspace*{0.5cm}
\forall\,{\underline{v}}\in D({\boldsymbol
  S^{\circledast}})=\underline{\cal V}({\Omega})\, 
$, where ${\underline{v}}\,=\,{\underline{v}}_{{\;\!}\mathfrak{c}}$ is written in Cartesian coordinates
and the vector Laplace operator $\Delta_{\underline{x} }$ acts as a
scalar on each component. 
We denote Friedrichs' extension of ${\boldsymbol S^{\circledast}}$ by 
${\boldsymbol S}:={\overline{\boldsymbol S^{\circledast}}}$, where
${\boldsymbol S}$ is defined on its domain 
$D({\boldsymbol S}):={\underline{\bb S}}_{{\;\!}}^{2}(.)=
{\underline{\bb W}}_{{\;\!}2}^{2}(.)\cap{\underline{\bb V}}(.)$ \,.} 
\end{definition}
\begin{remark} The range of ${\boldsymbol S}$ is $R({\boldsymbol
  S})={\underline{\bb H}}(.) $. In this context one may write
${\boldsymbol S} =-\Upsilon \Delta{\;\!}_{\underline{x}
}:{\underline{\bb S}}_{{\;\!}}^{2}(.) 
\,\longmapsto \,{\underline{\bb H}}(.)$. In Definition \ref{D5} one
can also use the Laplace operator in spherical polar
coordinates $\Delta_{r,\vartheta_{1},\dots\,\vartheta_{n-2},\varphi}$
(cf. Remark \ref{Lapldimn}). %like a scalar  instead of
                            %$\Delta_{\underline{x} }$, where
                            %${\underline{v}}\,=\,{\underline{v}}_{{\;\!}\mathfrak{c}}$
                            %has to be prescribed.
We avoid to
choose spherical polar
coordinates for both, $\Delta$
and  ${\underline{v}}\,=\,{\underline{v}}_{{\;\!}\mathfrak{s}}$, since
the vector Laplacian when  
applied to ${\underline{v}}_{{\;\!}\mathfrak{s}}\,$ produces
convoluted tensor fields in this combination 
(see \cite{MoonSpencer}).
\end{remark}
%%%%%%%%%%%
\noindent We sketch the fundamental properties of both operators (i.e. ${\boldsymbol L}$
as well as ${\boldsymbol S}$) using ${\boldsymbol S}$ as an example.
\begin{theorem} \label{thmstok}
The Stokes operator ${\boldsymbol S}$ is positive and self-adjoint.
Its inverse ${\boldsymbol S}^{-1}$ is injective, self-adjoint and compact.
\end{theorem}
\noindent The proof of Theorem \ref{thmstok} is a simple modification of 
Theorems 4.3 and 4.4 in \cite{CoFoi}. The essential tools are
the Rellich theorem and the Lax-Milgram lemma.
The well-known theorem of Hilbert (see, e.g. \cite{CouHil}) and
regularity results like \cite[Prop.~I.2.2]{Temam}
lead to more precise results, namely:
\begin{corollary}
\label{STOeiFU}
The Stokes operator only has a point spectrum.
 All eigenvalues $\lambda_{j}$
of  ${\boldsymbol S}$ are real and of finite multiplicity.
The associated eigenfunctions 
$\{{\underline{w}}_{j}({\underline{x}})\}_{j=1}^{\infty}$
% of the Stokes operator ${\boldsymbol S}$
(counted in multiplicity)
are an orthogonal basis of
${\underline{\bb H}}(.)$ and ${\underline{\bb V}}(.)$, i.e.
\begin{align*}
{{(a)}}&\quad {\boldsymbol 
S}{\underline{w}}_{j}:=\lambda_{j}{\underline{w}}_{j}\quad\mbox{for
}\quad{\underline{w}}_{j}\in D({\boldsymbol S})
\quad\forall\,j\in \NN\\
{ (b)}& \quad
0\,<\lambda_{1}\leq\,\lambda_{2}\,\leq\cdots\leq\,\lambda_{j}
\,\leq\cdots\quad\mbox{and}\quad 
\lim_{j\rightarrow\infty}\lambda_{j}=\infty
\\[-2mm]
{(c)}&\quad
\|{\underline{w}}_{j}\|_{{\underline{\bb 
H}}}\,=1\quad\forall \, j\in \NN\,.
\end{align*}
\end{corollary}
%%%%%%%%
\subsection{Eigenvalues and Eigenfunctions\label{sec_eigv}}
%%%%%%%%
\noindent 
We note, that formulas for the
complete sets of Laplace and  
Stokes eigenfunctions on the annuli ${\Omega}^{*}_{(n), \sigma}$ 
and on the unit ball ${\Omega}_{(n)}^{*}$ as subsets of ${\bb R}^{3}$
resp. ${\bb R}^{2}$ are derived in \cite{LeeRu}, \cite{RumKug1}, and 
\cite{RumTh2024}. The eigenvalues are obtained
as the squares of roots of transcendental
equations which are derived.
The transformation to the ${\Omega}_{(n), {\cal A}}$-domains then works as in \cite{RuRuTh2016}.
It is worth to note, that the roots of the derived transcendental equations behave like the roots of the Bessel functions:
Their consecutive zeros separate interdependently. 
In what follows we combine the Theorem from Subsection 6.4.4 in
\cite{Triebel} for the Laplace eigenfunctions  
on the unit ball ${\Omega}_{(n)}^{*}$ in ${\bb R}^{n}$ with new
ideas to obtain a solution of the Laplace 
 boundary value problem as a linear combination of the two linear independent solutions of the corresponding
Bessel's differential equation. We are going to sketch the ideas as well as the results below. 
We will use Bessel functions to study the eigenvalues and the Poincar\'e-constants 
for ${\cal A}\,\in\,(0,\infty)$ on
${\Omega}\,=\,{\Omega}_{(n), {\cal A}}$.
We employ the Bessel functions ${{{J_{k}}}(.)}$ and the Weber functions ${{{Y_{k}}}(.)}$ (cf. {{General notation B}})
of the order 
$k\,\in\,  {\bb N}_{o}$ if 
${\Omega}_{(n), {\cal A}}\subset {\bb R}^{n}$, where ${n}$ is an
even number, since the functions ${{{J_{k}}}(.)}$ and ${{{J_{-k}}}(.)}$ are linearly dependent in these
cases. The spherical
Bessel functions ${{{J_{k}}}(.)}$ of the first kind as well as 
the spherical Bessel functions  ${{{J_{-k}}}(.)}$ of the order 
{{$k\,\in\, \{\frac{1}{2}+m,\,m\,\in\, {\bb N}_{o}\} $}} are linearly
independent if
${\Omega}_{(n), {\cal A}}\subset {\bb R}^{n}$ and ${n}$ is odd.  They were used for 
${\Omega}_{(n), {\cal A}}\subset {\bb R}^{n}$ for ${n}\in\,{\bb N}$
odd.
Let us repeat some results of \cite{RuRuTh2016} and
\cite{RuRuTh2025}.% like a base clause for the extension to ${\Omega}_{(n),{\cal A}}\subset {\bb R}^{n}, \,n>3$.
\begin{note} \label{Result_n=2} Consider ${\Omega}_{(2),{\cal A}}\subset {\bb R}^{2}$. 
For arbitrary ${\cal A}\,\in\,(0,\infty)$ the first simple eigenvalue
${\lambda}_{1,L}({\cal A})$ of the 
Laplace (resp.~the first eigenvalue ${\lambda}_{1,S}({\cal A})$ of  the Stokes) operator
on ${\Omega}_{(2), {\cal A}}$
are the squares of the smallest positive solutions
$\kappa_{1,L}({\cal A})$ and 
$\kappa_{1,S}({\cal A})$, respectively,
of the transcendental equations (see \cite{RumZyl} and \cite{LeeRu})
\begin{align}\label{trans1n=2}
0&=\,J_{0}(\kappa_{L}({\cal A}){(1+{\cal
   A}/{2}}))Y_{0}(\kappa_{L}({\cal A}){{\cal A}/{2}}) 
\,-\,J_{0}(\kappa_{L}({\cal A}){{\cal
   A}/{2}})Y_{0}(\kappa_{L}({\cal A}){(1+{\cal A}/{2}}))\,\, \quad \, \text{ and } 
\\[.15cm]
\label{trans2n=2}
0&=\,J_{1}(\kappa_{S}({\cal A})({1+{\cal
   A}/{2
   }}))Y_{1}(\kappa_{S}({\cal A}){{\cal A}/{2}}) 
\,-\,J_{1}(\kappa_{S}({\cal A}){{\cal
   A}/{2}})Y_{1}(\kappa_{S}({\cal A})({1+{\cal A}/{2}}))\,.
\end{align}
After the transformation to ${\sigma}$ we find the equivalent roots
$\kappa_{1,L}({\sigma})$ and  
$\kappa_{1,S}({\sigma})$, resp., as
the smallest positive solutions of 
\begin{align}\label{trans1sigman=2}
0&=\,J_{0}(\kappa_{L}({\sigma}))Y_{0}(\kappa_{L}({\sigma}){\sigma})
\,-\,J_{0}(\kappa_{L}({\sigma}){\sigma})Y_{0}(\kappa_{L}({\sigma}))\,\, \quad \, \text{ and } 
\\[.15cm]
\label{trans2sigman=2}
0&=\,J_{1}(\kappa_{S}({\sigma}))Y_{1}(\kappa_{S}({\sigma}){\sigma})
\,-\,J_{1}(\kappa_{S}({\sigma}){\sigma})Y_{1}(\kappa_{S}({\sigma}))\,.
\end{align}
The conversion formulas between the roots are obvious for all ${\cal
  A}\,\in\,(0,\infty)$ and for the smallest ones we see  that
\begin{align}\label{kappa_A_sigma}
\kappa_{1,L}({\sigma})=({1+{\cal A}/{2}}) \kappa_{1,L}({\cal
  A})
\qquad \mbox{and}\qquad
\kappa_{1,S}({\sigma})  =({1+{\cal A}/{2}})
\kappa_{1,S}({\cal A})\,.
\end{align}
The eigenfunctions in ${\Omega}_{(2),{\cal A}}$ are %${{w}}_{{\;\!}1,L,{\cal A}}({\underline{x}})$ and ${\underline{w}}_{{\;\!}1,S,{\cal A}}({\underline{x}})$ in${\Omega}_{(2),{\cal A}}$  are given by
\begin{align}
{{w}}_{{\;\!}1,L,{\cal A}}({\underline{x}})&=\tilde{c}_{1,L,{\cal A}}\left(
J_{0}(\kappa_{1,L}({\cal
  A}) r)\,-\,\,
\frac{J_{0}({\textstyle{\frac{ \kappa_{1,L}({\cal A}) {\cal A}}{2}}})}
{Y_{0}({\textstyle{\frac{ \kappa_{1,L}({\cal A}) {\cal A}}{2}}})}
Y_{0}( \kappa_{1,L}({\cal A}) r
)\right) \,,\label{eiflaplace2} %\quad\quad \text{and}
\\
%%%%%%%%%%%%%%%%%%%%
{\underline{w}}_{{\;\!}1,S,{\cal A}}({\underline{x}})&={\tilde{c}_{1,S,{\cal A}}}
\left(
J_{1}(\kappa_{1,S}({\cal A})r)\,-\,
\frac{J_{1}({\textstyle{\frac{ \kappa_{1,S}({\cal A}) {\cal A}}{2}}})}
{Y_{1}({\textstyle{\frac{ \kappa_{1,S}({\cal A}) {\cal A}}{2}}})}
Y_{1}( \kappa_{1,S}({\cal A}) r
)
\right){\underline{\mathfrak{e}}}_{{\;\!}\varphi}\,.
\label{eifstokes2}
\end{align}
The numbers $\tilde{c}_{1,L,{\cal A}}$ and $\tilde{c}_{1,S,{\cal A}}$ are
scaling constants to ensure that the ${{\bb L}}_{{2}}(.)$- 
resp. ${\underline{\bb L}}_{{\;\!}{2}}(.)$-norms are $1$.
\end{note} 
%%%%%
%\noindent We turn to the case ${\Omega}_{(3),{\cal A}}$ in ${\bb R}^{3}$.
\begin{note} \label{Result_n=3} Consider ${\Omega}_{(3),{\cal A}}\subset {\bb R}^{3}$. 
For any ${\cal A}\,\in\,(0,\infty)$ the first simple eigenvalue
${\lambda}_{1,L}({\cal A})$ of the 
Laplace (resp.~the first eigenvalue ${\lambda}_{1,S}({\cal A})$ of the Stokes) operator
on ${\Omega}_{(3),{\cal A}}$ are the squares of the smallest positive solutions
$\kappa_{1,L}({\cal A})$ ($\kappa_{1,S}({\cal A})$, resp.)
of the transcendental equations 
(cf. \cite{RumTh2024})
\begin{align}
0&=\,J_{\frac{1}{2}}(\kappa_{L}({\cal A}){(1+{\cal
   A}/{2}}))J_{-\frac{1}{2}}(\kappa_{L}({\cal A}){{\cal A}/{2}}) 
\,-\,J_{\frac{1}{2}}(\kappa_{L}({\cal A}){{\cal
   A}/{2}})J_{-\frac{1}{2}}(\kappa_{L}({\cal A}){(1+{\cal A}/{2}}))\nonumber
\\
\label{trans1n=3}
&=\,\frac{2}{\pi \kappa_{L}({\cal A})} [({\cal A}/{2})(1+{\cal A}/{2})]^{-\frac{1}{2}}
\sin(\kappa_{L}({\cal A}))   , 
\\
\label{trans2n=3}
0&=\,J_{\frac{3}{2}}(\kappa_{S}({\cal A})({1+{\cal
   A}/{2
   }}))J_{-\frac{3}{2}}(\kappa_{S}({\cal A}){{\cal A}/{2}}) 
\,-\,J_{\frac{3}{2}}(\kappa_{S}({\cal A}){{\cal
   A}/{2}})J_{-\frac{3}{2}}(\kappa_{S}({\cal A})({1+{\cal A}/{2}}))\,. 
\end{align}
In \cite{RuRuTh2025} we have proved that for all ${\cal
  A}\,\in\,[0,\infty)$ the solution is $\kappa_{1,L}({\cal
  A})\,=\,\pi$.
After the transformation to ${\sigma}$ we find the equivalent solutions
$\kappa_{1,L}({\sigma})$ and  
$\kappa_{1,S}({\sigma})$ as
the smallest positive solutions of 
\begin{align}
0&=\,J_{\frac{1}{2}}(\kappa_{L}({\sigma}))J_{-\frac{1}{2}}(\kappa_{L}({\sigma}){\sigma})
\,-\,J_{\frac{1}{2}}(\kappa_{L}({\sigma}){\sigma})J_{-\frac{1}{2}}(\kappa_{L}({\sigma}))\,=\,
\frac{2}{\pi \kappa_{L}({\sigma})\sqrt{\sigma}}
\sin((1-{\sigma})\kappa_{L}({\sigma}))   , \label{trans1sigman=3}
\\
\label{trans2sigman=3}
0&=\,J_{\frac{3}{2}}(\kappa_{S}({\sigma}))J_{-\frac{3}{2}}(\kappa_{S}({\sigma}){\sigma})
\,-\,J_{\frac{3}{2}}(\kappa_{S}({\sigma}){\sigma})J_{-\frac{3}{2}}(\kappa_{S}({\sigma}))\,.
\end{align}
The conversion formulas are Equations \eqref{kappa_A_sigma} above.  
Now we choose the notation
$$
{\underline{\mathfrak{w}}}_{{\;\!}0}\,=\,\sin(\vartheta_1){\underline{\mathfrak{e}}}_{{\;\!}\varphi}\,, \quad{\underline{\mathfrak{w}}}_{{\;\!}-1}\,=\,cos(\varphi){\underline{\mathfrak{e}}}_{{\;\!}\vartheta_1}
-\sin(\varphi)\cos(\vartheta_1){\underline{\mathfrak{e}}}_{{\;\!}\varphi}\,,\quad {\mbox{and}} \,
\quad
{\underline{\mathfrak{w}}}_{{\;\!}1}\,=\,sin(\varphi){\underline{\mathfrak{e}}}_{{\;\!}\vartheta_1}
+\cos(\varphi)\cos(\vartheta_1){\underline{\mathfrak{e}}}_{{\;\!}\varphi}\,.\nonumber
$$
The eigenfunctions %${{w}}_{{\;\!}1,L,{\cal A}}({\underline{x}})$ and
                   %${\underline{w}}_{{\;\!}1,\alpha,S,{\cal
                   %A}}({\underline{x}})$ in ${\Omega}_{(3), {\cal
                   %A}}$
(here $\alpha\in\{-1,0,1\})$ are 
\begin{align}
{{w}}_{{\;\!}1,L,{\cal A}}({\underline{x}})&=\frac{\tilde{c}_{1,L,{\cal A}}}{\sqrt{r}}\left(
J_{\frac{1}{2}}(\pi r)\,-\,
\frac{J_{{\frac{1}{2}}}({\textstyle{\frac{ \pi {\cal A}}{2}}})}{J_{-{\frac{1}{2}}}({\textstyle{\frac{ \pi {\cal A}}{2}}})}
J_{-\frac{1}{2}}(\pi r )\right)\,, \label{eiflaplacen=3} %\quad\quad \text{and}
\\
%%%%%%%%%%%%%%%%%%%%
{\underline{w}}_{{\;\!}1,\alpha,S,{\cal A}}({\underline{x}})&=\frac{\tilde{c}_{1,\alpha,S,{\cal A}}}{\sqrt{r}}
\left(
J_{\frac{3}{2}}(\kappa_{1,S}({\cal A})r)\,-\,
\frac{J_{{\frac{3}{2}}}({\textstyle{\frac{ \kappa_{1,S}({\cal A}) {\cal A}}{2}}})}
{J_{-{\frac{3}{2}}}({\textstyle{\frac{ \kappa_{1,S}({\cal A}) {\cal A}}{2}}})}
J_{-\frac{3}{2}}( \kappa_{1,S}({\cal A}) r
)
\right)
{\underline{\mathfrak{w}}}_{{\;\!}\alpha}(\vartheta_1,\varphi)\,.
\label{eifstokesn=3}
\end{align} 
Again, $\tilde{c}_{1,L,{\cal A}}$ and $\tilde{c}_{1,\alpha,S,{\cal A}}$ are 
scaling constants. 
The notations 
${\underline{\mathfrak{w}}}_{{\;\!}\alpha}$ 
refer to the spherical surface harmonics (cf. Remark \ref{SurfHarm0+1}). 
Here $\alpha$ counts as above, i.e., $\alpha\in\{-1,0,1\}$. For the smallest eigenvalue 
of the Stokes operator  one needs
only one eigenfunction in order to calculate $c_{p,S}({\cal A})$. We select
${\underline{\mathfrak{w}}}_{{\;\!}0}$ (see also Note \ref{Result_n=2}).
\end{note}
\subsection{First Eigenvalues and Eigenfunctions for the Laplace Operator \label{sec_lap_n}}
Now we use ideas
from \cite{RuRuTh2016} and \cite{RuRuTh2025} to formulate
%%%%%%%%%%%5
\begin{theorem} \label{thm_Lapl_n} 
For all $n\,\in\,{\bb N},\,n>3$ and for any ${\cal
  A}\,\in\,(0,\infty)$ the first simple eigenvalue 
${\lambda}_{1,L}({\cal A})$ of the 
Laplace operator on ${\Omega}_{(n),{\cal A}}$ is the square of the smallest positive solution
$\kappa_{1,L}({\cal A})$  
of the following transcendental equations:  
\begin{align}\label{Lapl_Allgem_EW}
0\,=\,\left\{\begin{array}{ll}
J_{\frac{n}{2}-1}(\kappa_{L}({\cal A})(1+\frac{\cal A}{2}))Y_{\frac{n}{2}-1}(\kappa_{L}({\cal A}) \frac{\cal A}{2})
\,-\,J_{\frac{n}{2}-1}(\kappa_{L}({\cal A}) \frac{\cal A}{2})Y_{\frac{n}{2}-1}(\kappa_{L}({\cal A})(1+\frac{\cal A}{2})) 
& \forall \,
n\mbox{ even}\,\\[2mm]
J_{\frac{n}{2}-1}(\kappa_{L}({\cal A})(1+\frac{\cal A}{2}))J_{-\frac{n}{2}+1}(\kappa_{L}({\cal A}) \frac{\cal A}{2})
\,-\,J_{\frac{n}{2}-1}(\kappa_{L}({\cal A}) \frac{\cal A}{2})J_{-\frac{n}{2}+1}(\kappa_{L}({\cal A})(1+\frac{\cal A}{2})) & \forall \,
n\mbox{ odd}\,.
\end{array}
\right.
\end{align}
The corresponding first eigenfunction is
\begin{align}\label{Lapl_Allgem_EF}
{{w}}_{{\;\!}1,L,{\cal A}}({\underline{x}})
\,=\,\frac{\tilde{c}_{1,L,{\cal A}}}{r^{\frac{n}{2}-1}}
\left\{
\begin{array}{lr}
\Big(
J_{\frac{n}{2}-1}(\kappa_{1,L}({\cal A})r)\,-\,
\frac{J_{{\frac{n}{2}}-1}({\textstyle{\frac{ \kappa_{1,L}({\cal A}) {\cal A}}{2}}})}
{Y_{{\frac{n}{2}}-1}({\textstyle{\frac{ \kappa_{1,L}({\cal A}) {\cal A}}{2}}})}
Y_{\frac{n}{2}-1}( \kappa_{1,L}({\cal A}) r
)
\Big)& \forall \,
n\mbox{ even}\,,\\[2mm]
\Big (J_{\frac{n}{2}-1}(\kappa_{1,L}({\cal A})r)\,-\,
\frac{J_{{\frac{n}{2}}-1}({\textstyle{\frac{ \kappa_{1,L}({\cal A}) {\cal A}}{2}}})}
{J_{-{\frac{n}{2}}+1}({\textstyle{\frac{ \kappa_{1,L}({\cal A}) {\cal A}}{2}}})}
J_{-\frac{n}{2}+1}( \kappa_{1,L}({\cal A}) r
)
\Big)&\forall \,
n\mbox{ odd}\,,
\end{array}
\right.
\end{align}
where, again, 
$\tilde{c}_{1,L,{\cal A}}$ denote corresponding
scaling constants. % in the ${{\bb L}}_{{2}}(.)$-sense.
The functions in \eqref{Lapl_Allgem_EF} only 
depend on $r\,\in\,(\frac{\cal A}{2},1+\frac{\cal A}{2})$. The
corresponding spherical harmonic 
functions of degree $\ell = 0$ are the $S^{\{0\}}=1$, cf. \eqref{SphSurf0}.
\end{theorem}
%%%%%%%%%%%%%
\begin{remark} \label{thm_Lapl_sigma_n} It is easy to restate Theorem
  \ref{thm_Lapl_n} for $n>3$ (i.e. overall for $n>1$)
and ${\Omega}_{(n),\sigma}^{*}\subset {\bb R}^{n}$,
$\sigma\,\in\,(0,1)$ with the definition of $\sigma$ in \eqref{calA}
% \begin{equation*}
% {\sigma}\,=
% \frac{\displaystyle{\cal A}}{\displaystyle{ {\cal A} +2 }}\,\,
% \end{equation*} 
and the conversion formulas between the roots of the transcendental
equations \eqref{kappa_A_sigma}, namely,
\begin{align*}
\kappa_{1,L}({\sigma})=({1+{\cal A}/{2}}) \kappa_{1,L}({\cal
  A})\,.
\end{align*}
%Finally we rewrite the relations \eqref{Lapl_Allgem_EW} and \eqref{Lapl_Allgem_EF}.  
The details are given in Corollary \ref{col_sigma_Lapl_n}. 
\end{remark} 
%%%%%%%%%%%%%%
\noindent We quote the Theorem from Subsection 6.4.4 in \cite{Triebel} for the Laplace eigenfunctions
especially for the smallest eigenvalues and the first eigenfunctions
on ${\Omega}_{(n)}^{*}:=\{{\underline{x}} \in {\bb R}^{n}:\,
\|{\underline{x}}\|\,
<1\}$.
%%%%%%%%
\begin{note} \label{Trieb_Laplace_1}
Let $n>1$ and ${\Omega}_{(n)}^{*}\subset {\bb R}^{n}$ be the open unit ball. We denote the Laplace 
operator on %the unit ball
${\Omega}_{(n)}^{*}$ by ${\boldsymbol L}_{o}$.
% For all $n\,\in\,{\bb N},\,n>1$
The first simple eigenvalue
${\lambda}_{1,L}$ of %the Laplacian
${\boldsymbol L}_{o}$ on  ${\Omega}_{(n)}^{*}$ is the square of 
the smallest  positive solution
$\kappa_{1,L}$  of
\begin{align}\label{Lapl_einheitskugel_EW}
J_{\frac{n}{2}-1}(\kappa_{L}) =0\quad \quad \quad \quad \,\forall \,n\,\in\,{\bb N}:\,n>1\,.
\end{align}
The corresponding first eigenfunctions are
\begin{align}\label{Lapl_einheitskugel_EF}
{{w}}_{{\;\!}1,L}({\underline{x}})
\,=\,\frac{\tilde{c}_{1,L}}{r^{\frac{n}{2}-1}}
\cdot
J_{\frac{n}{2}-1}(\kappa_{1,L}\cdot r)\, ,
\end{align}
where, again, 
$\tilde{c}_{1,L}$ denote
scaling constants. % in the ${{\bb L}}_{{2}}(.)$-sense.
The functions in \eqref{Lapl_einheitskugel_EF}  only 
depend on $r\,\in\,(0,1)$ and are well-defined in $r=0$ as well. The corresponding spherical harmonic
functions of degree $\ell = 0$ are again $S^{\{0\}}=1$, cf. \eqref{SphSurf0}.
\end{note}
%%%%%%
\begin{proof} (of Thm.~\ref{thm_Lapl_n}) We use 
  \cite{Triebel} as quoted above in 
  Note \ref{Trieb_Laplace_1} and adapt the arguments of the proof in
  \cite{Triebel} for our case. 
For arbitrary ${\cal A}\,\in\,(0,\infty)$ and for $\ell = 0$  we replace the function
($J_{\ell+\frac{n}{2}-1}(\kappa_{L}(.)r)$  
by the linear combination $F(.)$ of the Bessel functions
$J_{\frac{n}{2}-1}$ and $J_{-\frac{n}{2}+1}$  
(resp. $Y_{\frac{n}{2}-1}$), where the second functions are finite for
${\cal A}>0$. The functions %$J_{-\frac{n}{2}+1}$
% (resp. $Y_{\frac{n}{2}-1}$)
$\{J_{\frac{n}{2}-1}, Y_{\frac{n}{2}-1}\}$ resp. $\{J_{-\frac{n}{2}+1},Y_{\frac{n}{2}-1}\}$
are linear independent solutions of the corresponding Bessel
differential equations. The
boundary conditions for $F(.)\,=\,F(r)$ on
$\partial{\Omega}_{(n),{\cal A}}$ are contained in
\eqref{Lapl_Allgem_EW}. 
In the second step we find $\kappa_{1,L}({\cal A})$ as the solutions of the
Bessel functions $J_{\frac{n}{2}-1}$. Now we are able to finish like
in step 3 
of Triebel's proof: The corresponding spherical harmonic functions of
degree $\ell = 0$ are  
the $S^{\{0\}}=1$. It ist worth to note, that the consecutive roots of
\eqref{Lapl_Allgem_EW} behave like the  
consecutive roots of the Bessel functions $J_{\frac{n}{2}-1}(.)$. This
property together with  
the eigenvalue $\ell (\ell+n-2)\,=\,0$ for $S^{\{0\}}=1$ at $\ell = 0$
(cf. Theorem \ref{thmsbeltr}) justify that in Theorem 
\ref{thm_Lapl_n} the smallest eigenvalue
${\lambda}_{1,L}({\cal A})$ of the 
Laplace operator on ${\Omega}_{(n),{\cal A}}$ is the square of the smallest positive solutions
$\kappa_{1,L}({\cal A})$ of \eqref{Lapl_Allgem_EW}. 
\end{proof}
%%%%%%%
\subsection{First Eigenvalues and Eigenfunctions for the Stokes
  Operator \label{sec_sto_n}}
%%%%%%%%%%%%%%%%%
\noindent Now for arbitrary ${\cal A}\,\in\,(0,\infty)$ and $n \geq 4$
we construct one eigenfunction for the smallest 
eigenvalue ${\lambda}_{1,S}({\cal A})$ 
of the Stokes operator. This is the easiest 
way for the calculation of the corresponding Poincaré
constant $c_{p,S}({\cal A})$. %So we are going to  benefit from a
                              %property of the Laplacian in
                              %application on a field
                              %${\underline{u}}_{\mathfrak{c}}$
                              %written in Cartesian coordinates. 
\begin{remark}  \label{Lapl_Trick} Let ${\Omega}\in\,{\bb R}^n$,
  $n>1$ be any of the domains 
${\Omega}_{(n),{\cal A}},\, {\Omega}^{*}_{(n), \sigma},$ and ${\Omega}_{(n)}^{*}$.
%The Laplacian (cf. Remark \ref{Lapldimn}) $\Delta$ act as a scalar in
%the application  on
%${\underline{u}}_{\mathfrak{c}}\,\in\,{\underline{C}}^2({\Omega})$,
%cf. Notation \label{N3}. This is also true for the Laplacian
%$\Delta\,=\,\Delta_{r,\vartheta_{1},\dots\,\vartheta_{n-2},\varphi}$
%in spherical polar coordinates.
As shown for $n=3$ in \cite{RumHab} %it holds that
\begin{equation}\label{Lapl_u} 
\Delta {\underline{u}}\,=\,\Delta (
\sum_{j=1}^{n} u_{j,{\mathfrak{c}}}{\underline{\mathfrak{e}}}_{j})
		\,=\,
\sum_{j=1}^{n} (\Delta \,u_{j,{\mathfrak{c}}}){\underline{\mathfrak{e}}}_{j}\,=\,
		\left[
		\begin{array}{l}
\Delta_{r,\vartheta_{1},\dots\,\vartheta_{n-2},\varphi}( u_{1,{\mathfrak{c}}})\\
\quad\quad {~}\vdots\\
\Delta_{r,\vartheta_{1},\dots\,\vartheta_{n-2},\varphi}( u_{n,{\mathfrak{c}}})
		\end{array}
		\right] \,.
\end{equation}
\end{remark} 
\noindent 
We cite a Lemma from \cite{RumHab} (resp. \cite{RumTh2024}).
\begin{lem} \label{LE1}
Let $\,{\Omega}_{(n)}^{*}$, $n\,\in\,{\bb N}, n>1$, be the open unit
ball in ${\bb R}^{n}$. %,${\underline{x}}\,\in\,{{\Omega}_{(n)}^{*}}$ -
                     %$r$ is the radius in a spherical coordinate system.
Then for all $r\,\in \, (0,1)$ there exists no solenoidal vector function 
${\underline{g}} \in (C^{1}({\Omega}_{(n)}^{*}))^{n}$, which is only depending of 
the variable $r$, with ${\underline{g}}(r)\,\neq\,{\underline{c}}$. Here ${\underline{c}}$ denotes any 
constant vector.
\end{lem} 
\begin{note} \label{LE1_All}
It is easy to see, that
Lemma \ref{LE1} is also true  for all ${\Omega}^{*}_{(n),\sigma}$ and ${\Omega}^{}_{(n),\cal A}$
in ${\bb R}^{n}$.
\end{note}
%%%%%%%%%%%%%
\begin{corollary} \label{Stokes_sph_harm_l=1}
  Any eigenfunction ${\underline{w}}_{\mathfrak{c}}$ of the 
Stokes operator on ${\Omega}^{*}_{(n)}$,  
${\Omega}^{*}_{(n),\sigma}$ or 
${\Omega}^{}_{(n),\cal A}$ (written in Cartesian coordinates) has to
possess a function of the surface harmonic function of 
the degree $\ell  \geq 1$ as a factor in at least one of its components. 
\end{corollary}
%%%%%%%%
\noindent In what follows we change our view on the outcomes in the
Notes \ref{Result_n=2} and \ref{Result_n=3} 
to get another look at the first Stokes eigenfunctions on
${\Omega}^{*}_{(n)}$,  
${\Omega}^{}_{(n), \cal A} $ and ${\Omega}^{*}_{(n), \sigma}$.  
We formulate our completely new result in the subsequent Theorem \ref{thm_Stokes_n} for 
${\Omega}^{}_{(n), \cal A}$. The structural design of the
selected first Stokes eigenfunctions in  
Theorem \ref{thm_Stokes_n} is copied from the toroidal fields in
${\Omega}^{*}_{(3)}$ (cf. \cite{RumTh2024}). 
Toroidal fields are vector fields with vanishing divergence in the setting 
${\underline{w}}_{{\;\!}\mathfrak{s}}\,=\, {\underline{0}} + w_{\varphi,{\mathfrak{s}}}{\underline{\mathfrak{e}}}_{\varphi}$ 
or ${\underline{w}}_{{\;\!}\mathfrak{s}}\,=\,0\cdot
{\underline{\mathfrak{e}}}_{r}+
 w_{\vartheta_{1},{\mathfrak{s}}}{\underline{\mathfrak{e}}}_{\vartheta_{1}}
+
 w_{\varphi,{\mathfrak{s}}}{\underline{\mathfrak{e}}}_{\varphi}$.
The translation of our result for the Stokes eigenfunctions on ${\Omega}^{*}_{(n)}$ and
${\Omega}^{}_{(n),\sigma}$ is obvious and formulated in 
two propositions afterwards. 
In view of Theorem \ref{thm_Stokes_n}
one can regard the first Stokes eigenfunctions in Note
\ref{Result_n=2} as a special case. 
The main steps in the proof of the Theorem are the shape of the first
Stokes eigenfunctions and then using Remark \ref{Lapl_Trick} and Corollary
\ref{Stokes_sph_harm_l=1}.
%%%%%%%%%
\begin{theorem} \label{thm_Stokes_n} Let $n>3$ {{(resp. $n>1$)}} and ${\Omega}_{{\cal A}}={\Omega}_{(n),{\cal A}}\subset {\bb R}^{n}$.
For any ${\cal A}\,\in\,(0,\infty)$ the first eigenvalue
${\lambda}_{1,S}({\cal A})$ of the Stokes operator on ${\Omega}_{(n),{\cal A}}$ is %for all $n\,\in\,{\bb N},\,n>3$ 
the square of the smallest positive solutions
$\kappa_{1,S}({\cal A})$
of the transcendental equations
\begin{align}\label{Stokes_Allgem_EW}
0\,=\,\left\{\begin{array}{ll}
J_{\frac{n}{2}}(\kappa_{S}({\cal A})(1+\frac{\cal A}{2}))Y_{\frac{n}{2}}(\kappa_{S}({\cal A}) \frac{\cal A}{2})
\,-\,J_{\frac{n}{2}}(\kappa_{S}({\cal A}) \frac{\cal A}{2})Y_{\frac{n}{2}}(\kappa_{S}({\cal A})(1+\frac{\cal A}{2})) 
& \forall \,
n\mbox{ even} \,\\[2mm] 
%{~} & {~} \\
J_{\frac{n}{2}}(\kappa_{S}({\cal A})(1+\frac{\cal A}{2}))J_{-\frac{n}{2}}(\kappa_{S}({\cal A}) \frac{\cal A}{2})
\,-\,J_{\frac{n}{2}}(\kappa_{S}({\cal A}) \frac{\cal A}{2})J_{-\frac{n}{2}}(\kappa_{S}({\cal A})(1+\frac{\cal A}{2})) & \forall \,
n\mbox{ odd}\,.
\end{array}
\right.\,.
\end{align}
The equations \eqref{Stokes_Allgem_EW} are derived from the homogeneous Dirichlet boundary conditions 
on $\partial{\Omega}_{(n),{\cal A}}$. 
One of the corresponding first eigenfunctions are the fields
\begin{align}\label{Stokes_Allgem_EF}
{\underline{w}}_{{\;\!}\varphi,S,{\cal A}}({\underline{x}})\,:=\,{\tilde{c}_{1,S,{\cal A}}}
\left\{
\begin{array}{ll}
\, \varrho(r) \sin{\vartheta_{1}}\sin{\vartheta_{2}}\cdots \sin{\vartheta_{n-2}}\cdot 
{\underline{\mathfrak{e}}}_{{\;\!}\varphi}\,& \forall \,n\,\in\,{\bb N}\,;\,n>2  \\
{~} & {~} \\ 
\, \varrho(r) \cdot 
{\underline{\mathfrak{e}}}_{{\;\!}\varphi}\,& 
\text{for }\,
\,n=2 \, 
\end{array}\right.
\end{align}
(cf. Corollary \ref{Stokes_sph_harm_l=1} and the corresponding
spherical harmonics of the degree  
$\ell = 1$ given in Eq. \eqref{SphSurf1}).
Again, $\tilde{c}_{1,S,{\cal A}}$ denote the corresponding  
scaling constants.
The functions $\varrho(r)$ only depend on $r\,\in\,(\frac{\cal
  A}{2},1+\frac{\cal A}{2})$. They are% specified as: 
\begin{align}\label{Stokes_Allgem_Rho}
\varrho(r)
\,:=\,\frac{1}{r^{\frac{n}{2}-1}}
\left\{
\begin{array}{ll}
\big(
J_{\frac{n}{2}}(\kappa_{1,S}({\cal A})r)\,-\,
\frac{J_{{\frac{n}{2}}}({\textstyle{\frac{ \kappa_{1,S}({\cal A}) {\cal A}}{2}}})}
{Y_{{\frac{n}{2}}}({\textstyle{\frac{ \kappa_{1,S}({\cal A}) {\cal A}}{2}}})}
Y_{\frac{n}{2}}( \kappa_{1,S}({\cal A}) r)
\big)&\forall \,
n\mbox{ even}\,,\\[2mm]
%{~} & {~} \,\\
\big(J_{\frac{n}{2}}(\kappa_{1,S}({\cal A})r)\,-\,
\frac{J_{{\frac{n}{2}}}({\textstyle{\frac{ \kappa_{1,S}({\cal A}) {\cal A}}{2}}})}
{J_{-{\frac{n}{2}}}({\textstyle{\frac{ \kappa_{1,S}({\cal A}) {\cal A}}{2}}})}
J_{-\frac{n}{2}}( \kappa_{1,S}({\cal A}) r
)
\big)&\forall \,
n\mbox{ odd}\,. 
\end{array}
\right.
\end{align}
\end{theorem} 
\noindent Before we prove the Theorem we illustrate our %show the application of our
method for the cases $n=2$ and $n=3$.
%%%%%%%%%%%%%%%%%%%%%%%%%%%%%%%
\begin{example} \label{Idee_Stokes_n=2} For arbitrary ${\cal
    A}\,\in\,(0,\infty)$ we consider the first Stokes eigenfunction on 
${\Omega}_{(2),{\cal A}}\subset {\bb R}^{2}$ 
(cf. Eq. \eqref{eifstokes2} and \cite{RuRuTh2016}). % to demonstrate the idea for the proof of Theorem \ref{thm_Stokes_n}. 
We write using $\varrho(r)$ as definied in \eqref{Stokes_Allgem_Rho} %as an abbreviation:
\begin{align}
{\underline{w}}_{\mathfrak{s}}({\underline{x}})\,:=\,
{\underline{w}}_{{\;\!}1,S,{\cal A}}({\underline{x}}) &\,=\,
\tilde{c}_{1,S,{\cal A}}
\, \varrho(r) \cdot 
{\underline{\mathfrak{e}}}_{{\;\!}\varphi}\,=\,\tilde{c}
\, \varrho(r) \cdot 
{\underline{\mathfrak{e}}}_{{\;\!}\varphi}\,.
\nonumber
\end{align}
The vector ${\underline{w}}_{\mathfrak{s}}({\underline{x}})$ is
divergence free since it is a Stokes eigenfunction.  
This agrees with the use of \eqref{DIV=2} in Remark \ref{Div_phi} 
since it also shows
$\text{div}{\;\!}{\underline{w}}_{\mathfrak{s}}\,=\,0$. 
The homogeneous Dirichlet boundary conditions 
on $\partial{\Omega}_{(2),{\cal A}}$ are fulfilled due to \eqref{Stokes_Allgem_EW}. 		
For applying Corollary \ref{Stokes_sph_harm_l=1} we have to 
transform ${\underline{w}}_{\mathfrak{s}}$ 
to Cartesian coordinates.
This %e transformation \eqref{R1_T=2} of
% ${\underline{w}}_{\mathfrak{s}}$ in the Cartesian coordinates
yields
\begin{equation} 
{\underline{w}}_{\mathfrak{c}}\,=\,
{\underline{\underline{T}}}_{{\mathfrak{c}},{\mathfrak{s}}}{\underline{w}}_{\mathfrak{s}}\,
=\,\tilde{c}
\, 
\left[
		\begin{array}{l}
- \varrho(r)\sin \varphi\\
\,\,\varrho(r)\cos \varphi
		\end{array}
		\right] \,,
\end{equation}
where the columns are coordinates.  The functions $\sin \varphi$ and  $\cos \varphi$
are surface harmonic function of
degree $\ell  = 1$. With Remark \ref{Lapl_Trick}  we arrive at
\begin{equation} \label{EWS_2}
-\,\Delta
{\underline{w}}_{\mathfrak{c}}\,=\,
- \tilde{c}\cdot
\left[
		\begin{array}{l}
- \Delta (\varrho(r)\sin \varphi)\\
\,\,\,\,\Delta(\varrho(r)\cos \varphi)
		\end{array}
		\right] \,=\,\tilde{c}\cdot (\kappa_{1,S}({\cal A}))^2
\left[
		\begin{array}{l}
- \varrho(r)\sin \varphi\\
\,\,\varrho(r)\cos \varphi
		\end{array}
		\right]\,=\,(\kappa_{1,S}({\cal A}))^2{\underline{w}}_{\mathfrak{c}}\,.
\end{equation}
The Bessel functions $J_{1}$ and $Y_{1}$ in \eqref{Stokes_Allgem_Rho}
for $n=2$ are corresponding to the surface harmonic
functions of 
the degree $\ell  = 1$. The calculations have to be carried out along the
lines of the proof of the Theorem in 
Subsection 6.4.4 in \cite{Triebel}.
{{For $n=2$ in Step 3 of this proof %for the proof of the Theorem in
    % Subsection 6.4.4
    for the first component
in \eqref{EWS_2} we obtain 
\begin{align} \label{Rechnen_n_2}
-\,\Delta
{{w}}_{1,\mathfrak{c}}\,=\,\tilde{c}\cdot
\Delta\,(\varrho(r)\sin \varphi)\,=\,\tilde{c}\cdot\,\left(
\frac{1}{r}
\frac{\text{d}}{\text{d} r}({r}\frac{\text{d} {\varrho(r)}}{\text{d}  r})\sin \varphi\,+\,
\frac{{\varrho(r)}}{r^2}
\frac{\text{d}^{{\;\!}2}}{{\text{d}} \varphi^2} (\sin \varphi) \right)\,=\,\nonumber\\
=\,\tilde{c}\sin \varphi\,\left(\frac{\text{d}^{{\;\!}2}}{{\text{d}} r} + \,\frac{1}{r}
\frac{\text{d}}{\text{d} r}\,-\,\frac{1}{r^2}\right){\varrho(r)}\,=\,-(\kappa_{1,S}({\cal A}))^2 \tilde{c}\cdot (\varrho(r)\sin \varphi)\,\,,
\end{align}
where in \eqref{Rechnen_n_2} we see the Bessel differential equation
with the solutions $J_{1}$ and $Y_{1}$.  
}}
To sum up, we have shown also directly, that ${\underline{w}}_{{\;\!}1,S,{\cal A}}({\underline{x}})$
is an eigenfunction of the Stokes operator on ${\Omega}_{(2),{\cal A}}$ and that $\lambda_{1,S}({\cal A})\,=\,
(\kappa_{1,S}({\cal A}))^2$ is the corresponding first eigenvalue.
\end{example}
%%%%%%%%%%%%%
\begin{example} \label{Idee_Stokes_n=3} In Note \ref{Result_n=3} we selected 
${\underline{\mathfrak{w}}}_{{\;\!}0}\,=\,\sin (\vartheta_1) \,
{\underline{\mathfrak{e}}}_{{\;\!}\varphi}$. Now we
choose exactly one of the first Stokes eigenfunctions on
${\Omega}_{{\cal A}}={\Omega}_{(3),{\cal A}}\subset {\bb R}^{3}$.
For arbitrary ${\cal A}\,\in\,(0,\infty)$ (cf. Eq. \eqref{eifstokesn=3}) to demonstrate Theorem \ref{thm_Stokes_n}. 
Using \eqref{Stokes_Allgem_Rho} we write
\begin{align}
{\underline{w}}_{\mathfrak{s}}({\underline{x}})\,:=\,
{\underline{w}}_{{\;\!}1,S,{\cal A}}({\underline{x}})& \,=\,\tilde{c}_{1,S,{\cal A}}
\, \varrho(r) \cdot {\underline{\mathfrak{w}}}_{{\;\!}0}\,=\,
\tilde{c}
\, \varrho(r) \cdot \sin (\vartheta_1) 
{\underline{\mathfrak{e}}}_{{\;\!}\varphi}\,.
\nonumber
\end{align}
One easily calculates that 
% ${\underline{w}}_{\mathfrak{s}}({\underline{x}})$ we get
$\text{div}{\;\!}{\underline{w}}_{\mathfrak{s}}\,=\,0$.
(cf. Remark \ref{Div_phi} esp. Eq. \eqref{DIV=3}).
The homogeneous Dirichlet boundary conditions 
on $\partial{\Omega}_{(3),{\cal A}}$ are fulfilled by \eqref{Stokes_Allgem_EW}. 		
For applying Corollary \ref{Stokes_sph_harm_l=1} we transform
${\underline{w}}_{\mathfrak{s}}$ 
into Cartesian coordinates.
This yields
\begin{equation} 
{\underline{w}}_{\mathfrak{c}}\,=\,
{\underline{\underline{T}}}_{{\mathfrak{c}},{\mathfrak{s}}}{\underline{w}}_{\mathfrak{s}}\,
=\,\tilde{c}\varrho(r) \sin \vartheta_1 \,
\, 
\left[
		\begin{array}{l}
- \sin \varphi\\
\,\, \,\cos \varphi \\
\quad 0
		\end{array}
		\right] \,,
\end{equation}
where, again, the columns are coordinates. We indicate that the
functions $S^{\{1\}}_2\,=\,\sin \vartheta_1 \sin \varphi$ and \,  
$S^{\{1\}}_1\,=\,\sin \vartheta_1 \cos \varphi$\,
are surface harmonic functions of
degree $\ell  = 1$ in ${\bb R}^{3}$.
We use Remark \ref{Lapl_Trick} and proceed analogously to the
computations in Example \ref{Idee_Stokes_n=2}  to find
\begin{equation} \label{EWS_3}
-\,\Delta
{\underline{w}}_{\mathfrak{c}}\,=\,
- \tilde{c}\cdot
\left[
		\begin{array}{l}
- \Delta \varrho(r)\sin \vartheta_1 \sin \varphi\\
\,\,\Delta\varrho(r)\sin \vartheta_1 \cos \varphi\\
\quad \quad 0
		\end{array}
		\right] \,=\,\tilde{c}\cdot (\kappa_{1,S}({\cal A}))^2
\left[
		\begin{array}{l}
- \varrho(r)\sin \vartheta_1 \sin \varphi\\
\,\,\varrho(r)\sin \vartheta_1 \cos \varphi\\
\quad \quad 0
		\end{array}
		\right]\,=\,(\kappa_{1,S}({\cal A}))^2{\underline{w}}_{\mathfrak{c}}\,.
\end{equation}
The corresponding Bessel functions to the surface harmonic functions of
the degree $\ell  = 1$ in \eqref{Stokes_Allgem_Rho} are $J_{\frac{3}{2}}$ and $J_{-\frac{3}{2}}$.
We refer to the proof of the Theorem in
Subsection 6.4.4 in in \cite{Triebel} for the calculations.
It is quite evident, that 
${\underline{w}}_{{\;\!}1,S,{\cal A}}({\underline{x}})\,=\,\tilde{c}
\, \varrho(r) \cdot {\underline{\mathfrak{w}}}_{{\;\!}0}$
is an eigenfunction of the Stokes operator on ${\Omega}_{(3),{\cal A}}$ to the first eigenvalue $\lambda_{1,S}({\cal A})\,=\,
(\kappa_{1,S}({\cal A}))^2$ .
\end{example}
\begin{proof} (of Thm.~\ref{thm_Stokes_n}) We follow the path already
  chosen in  Examples \ref{Idee_Stokes_n=2} and \ref{Idee_Stokes_n=3}.
We only have to prove our Theorem for $n \geq 4 $.
In what follows let us choose $n=4$ as an illustration together
with the general case $n=n$. 
We write the eigenfunctions with the abbreviations in
\eqref{Stokes_Allgem_Rho}
%and $\tilde{c}_{1,S,{\cal A}}\,=\,\tilde{c}$ 
as 
\begin{align}\label{Proof_Stokes_Allgem_EF}
{\underline{w}}_{\mathfrak{s}}({\underline{x}})\,=\,
{\underline{w}}_{{\;\!}1,S,{\cal A}}({\underline{x}}) \,:=\,{\tilde{c}_{1,S,{\cal A}}}
\left\{
\begin{array}{ll}
\, \varrho(r) \sin{\vartheta_{1}}\sin{\vartheta_{2}}\cdots \sin{\vartheta_{n-2}}\cdot 
{\underline{\mathfrak{e}}}_{{\;\!}\varphi}\,& \forall \,n\,\in\,{\bb N}\,,\,n>4 \\
{~} & {~} \\ 
\, \varrho(r) \,\sin{\vartheta_{1}}\sin{\vartheta_{2}}\cdot 
{\underline{\mathfrak{e}}}_{{\;\!}\varphi}\,& \text{for } \,n=4\,\,.
\end{array}\right.
\end{align}
From now on we abbreviate $\tilde{c}_{1,S,{\cal A}}\,=\,\tilde{c}$.
The functions in \eqref{Proof_Stokes_Allgem_EF} are of the structure 
${\underline{w}}_{\mathfrak{s}}({\underline{x}})\,=\,F(.)\cdot 
{\underline{\mathfrak{e}}}_{{\;\!}\varphi}$, where $F(.)$  does not
depend on $\varphi$. So we see by Remark \ref{Div_phi} (especially
Eqs. \eqref{DIV=4} and \eqref{DIV=n}), that the divergence of the
functions  
${\underline{w}}_{\mathfrak{s}}({\underline{x}})$ vanishes. 
It is easy to check, that homogeneous Dirichlet boundary conditions
on $\partial{\Omega}_{(4),{\cal A}}$
resp. $\partial{\Omega}_{(n),{\cal A}}$ hold due to \eqref{Stokes_Allgem_EW}.\\[1mm]
Now we transform ${\underline{w}}_{\mathfrak{s}}$
into Cartesian coordinates and see with %Corollary \ref{Stokes_sph_harm_l=1}.
Remark \ref{R1} and the relations \eqref{R1_T=4} and \eqref{R1_T=n} that
\begin{align*} 
{\underline{w}}_{\mathfrak{c}}\,=\,
{\underline{\underline{T}}}_{{\mathfrak{c}},{\mathfrak{s}}}{\underline{w}}_{\mathfrak{s}}\,
=\,&
\tilde{c}\varrho(r) \sin \vartheta_1 \sin \vartheta_2\,
\, 
\left[
		\begin{array}{l}
- \sin \varphi\\
\,\, \,\cos \varphi \\
\quad 0 \\
\quad 0
		\end{array}
		\right] \, & \text{for } n=4\,\,\text{and }\\
{\underline{w}}_{\mathfrak{c}}\,=\,
{\underline{\underline{T}}}_{{\mathfrak{c}},{\mathfrak{s}}}{\underline{w}}_{\mathfrak{s}}\,
=\,&
\tilde{c}\varrho(r) \sin \vartheta_1 \sin{\vartheta_{2}}\cdots \sin{\vartheta_{n-2}}\,
\left[
		\begin{array}{l}
- \sin \varphi\\
\,\, \,\cos \varphi \\
\quad 0 \\
\quad  \,  \vdots\\
\quad 0
		\end{array}
		\right] \,& \quad \text{for all } \,n\,\in\,{\bb N}\,,\,n>4 \,\, .
\end{align*}
Again, in what follows columns are coordinates. We only have to consider the two
non-vanishing components $w_{1,{\mathfrak{c}}}$ and $w_{2,{\mathfrak{c}}}$ in the component-by-component view:
${\underline{w}}_{\mathfrak{c}}\,=\,
w_{1,{\mathfrak{c}}}\cdot{\underline{\mathfrak{e}}}_{1}\,+\,w_{2,{\mathfrak{c}}}\cdot{\underline{\mathfrak{e}}}_{2}
\,+\,{\underline{0}}$.
%Let us move on to the general topic ${\bb R}^{n},\,n\,\geq 4$.
We denote the surface harmonic functions of the degree $\ell  = 1$ for $n\,\geq 4$ (cf.
Notation \ref{SurfHarm0+1}) by 
$S^{\{1\}}_2\,=\,\sin \vartheta_1 \sin{\vartheta_{2}}\cdots \sin{\vartheta_{n-2}} \sin \varphi$ \,\,and \,\,
$S^{\{1\}}_1\,=\,\sin \vartheta_1 \sin{\vartheta_{2}}\cdots \sin{\vartheta_{n-2}} \cos \varphi$.
Now we have made all the preparations to follow component-by-component
some steps  in the  
proof of the Theorem from Subsection 6.4.4 in \cite{Triebel}. One only
has to replace the 
functions $\frac{1}{r^{\frac{n}{2}-1}}J_{\ell+\frac{n}{2}-1}({{\kappa}} \,r)$ for $\ell = 1$ 
by the functions $\varrho(r)$ from \eqref{Stokes_Allgem_Rho} and for
any ${\cal A}\,\in\,(0,\infty)$ to
replace ${{\kappa}}\,=\,{{{\mu}_{k}^{\ell+\frac{n}{2}-1}}}$
by $\kappa_{1,S}({\cal A})$. % (i.e. for any ${\cal A}\,\in\,(0,\infty)$
                           % the $\kappa_{1,S}({\cal A})$ in contrast
                           % to
                           % \cite{Triebel}:${\blue{\kappa}}\,=\,{\blue{{\mu}_{k}^{\ell+\frac{n}{2}-1}}}$ )
We note, that in the functions $\varrho(r)$ we have the weight
$\frac{1}{r^{\frac{n}{2}-1}}$ as 
a factor of a linear combination of
the Besselfunctions $J_{\frac{n}{2}}$ and $J_{-\frac{n}{2}}$ 
resp. $Y_{\frac{n}{2}}$, where the second functions are finite for ${\cal A}>0$. 
The functions $\varrho(r)$ from \eqref{Stokes_Allgem_Rho} with the Besselfunctions $J_{\frac{n}{2}}$ and $J_{-\frac{n}{2}}$ 
resp. $Y_{\frac{n}{2}}$ of degree $\ell = 1$ are the 
perfect match
resp. the perfect couple 
in view on the corresponding spherical harmonic functions of degree $\ell = 1$:\,
$S^{\{1\}}_1$\,and \, $S^{\{1\}}_2$.
Let $$\underline{h}:=(- \varrho(r) \sin \vartheta_1 \sin{\vartheta_{2}}\cdots \sin{\vartheta_{n-2}} \sin \varphi\,,
\,\,\varrho(r)\sin \vartheta_1 \sin{\vartheta_{2}}\cdots
\sin{\vartheta_{n-2}} \cos \varphi\,,
0 \,,\dots\,,0)^T\,.$$
With this abbreviation applying Remark \ref{Lapl_Trick} we arrive at
\begin{equation*}
- \Delta
{\underline{w}}_{\mathfrak{c}}\,=\,
- \tilde{c}\cdot \Delta \underline{h}
% \left[
% 		\begin{array}{r}
% - \varrho(r) \sin \vartheta_1 \sin{\vartheta_{2}}\cdots \sin{\vartheta_{n-2}} \sin \varphi\\
% \,\,\varrho(r)\sin \vartheta_1 \sin{\vartheta_{2}}\cdots \sin{\vartheta_{n-2}} \cos \varphi\\
% \quad \quad 0\\
% \quad \quad \vdots \\
% \quad \quad 0
% \end{array}
% 		\right] 
  \,=\,
		\tilde{c}\cdot (\kappa_{1,S}({\cal A}))^2 \underline{h}
% \left[
% \begin{array}{r}
% 		- \varrho(r) \sin \vartheta_1 \sin{\vartheta_{2}}\cdots \sin{\vartheta_{n-2}} \sin \varphi\\
% \,\,\varrho(r)\sin \vartheta_1 \sin{\vartheta_{2}}\cdots \sin{\vartheta_{n-2}} \cos \varphi\\[.2cm]
% \quad \quad 0\\
% \quad \quad \vdots \\
% \quad \quad 0
% \end{array}
% 		\right]\,=\,
		= (\kappa_{1,S}({\cal A}))^2{\underline{w}}_{\mathfrak{c}}\,=\,
		{\lambda}_{1,S}({\cal A})\,{\underline{w}}_{\mathfrak{c}}\,\,,
\end{equation*}
\noindent  where, again, we reason as in the step 3
of Triebel's proof in \cite{Triebel}.
The smallest positive solutions
$\kappa_{1,S}({\cal A})$ of \eqref{Stokes_Allgem_EW} now play the same role as the first 
root ${{{\mu}_{1}^{\ell+\frac{n}{2}-1}}}$ of the Bessel 
functions $J_{\ell+\frac{n}{2}-1}(.)\,=\,J_{\frac{n}{2}}(.)\,$ at $\ell = 1$ 
in \cite{Triebel}.
We complete with the following arguments:
The degree $\ell = 1$ for the spherical harmonic functions is the smallest possible degree
taking into account Corollary \ref{Stokes_sph_harm_l=1}. The functions $\varrho(r)$ from \eqref{Stokes_Allgem_Rho}
have to be composed of the corresponding Bessel functions at $\ell = 1$.
We have the degree $\ell = 1$ as the smallest possible degree and the smallest positive solutions
$\kappa_{1,S}({\cal A})$ of \eqref{Stokes_Allgem_EW} and for this reason with ${\lambda}_{1,S}({\cal A})$
the  smallest eigenvalue
of the Stokes operator on ${\Omega}_{(n),{\cal A}}$.
\end{proof}
%%%%%%%%%%%%%%%%%%%%%%%%%%%%%%%%
\noindent We formulate the corresponding result for the Stokes operator
on the open unit ball in ${\bb R}^{n}$.
%%%%%%%%%%%%%%
\begin{corollary} \label{Stokes_Unit_ball}
Let ${\bb R}^{n}, n>1$, and ${\Omega}_{(n)}^{*}\subset {\bb R}^{n}$ be the open unit ball. We denote the Stokes 
operator on 
${\Omega}_{(n)}^{*}$ by ${\boldsymbol S}_{o}$.
For any $n\,\in\,{\bb N},\,n>1$, the first simple eigenvalue
${\lambda}_{1,S}$ of  %the Stokes operator
${\boldsymbol S}_{o}$ %on  ${\Omega}_{(n)}^{*}$
is the square of 
the smallest  positive solution
$\kappa_{1,S}$  of
\begin{align}\label{Stokes_einheitskugel_EW}
J_{\frac{n}{2}}(\kappa_{S}) =0\quad \quad \quad \quad \,\forall \,n>1\,.
\end{align}
For any $n>1$ one of the corresponding first eigenfunctions are the vector fields
\begin{align}\label{Stokes_Allgem_Einsball_EF}
{\underline{w}}_{{\;\!}\varphi,S}({\underline{x}})\,:=\,{\tilde{c}_{1,S}}
\left\{
\begin{array}{ll}
\, \varrho(r) \sin{\vartheta_{1}}\sin{\vartheta_{2}}\cdots \sin{\vartheta_{n-2}}\cdot 
{\underline{\mathfrak{e}}}_{{\;\!}\varphi}\,& \forall \,n\,\in\,{\bb N}\,;\,n>2  \\
{~} & {~} \\ 
\, \varrho(r) \cdot 
{\underline{\mathfrak{e}}}_{{\;\!}\varphi}\,& 
\text{for }\,
\,n=2 \, ,
\end{array}\right.
\end{align}
where the $\tilde{c}_{1,S}$ denote the corresponding 
scaling constants. % in the ${\underline{\bb L}}_{{2}}(.)$-sense.\\
The $\varrho(r)$ are scalar functions of \,$r\,\in\,(0,1)$ defined as
\begin{align}\label{Stokes_Allgem_Rho_Einsball}
\varrho(r)
\,:=\,\frac{1}{r^{\frac{n}{2}-1}}
J_{\frac{n}{2}}(\kappa_{1,S} \,r)
&\,\quad\quad\forall \,
n\,\in\,{\bb N},\,n >1. 
\end{align}
Especially one has \quad ${\boldsymbol S}_{o}\,{\underline{w}}_{{\;\!}\varphi,S}\,=\,
{\lambda}_{1,S}\,{\underline{w}}_{{\;\!}\varphi,S}\,=\,(\kappa_{1,S})^2\,{\underline{w}}_{{\;\!}\varphi,S}\,.$
All is well-defined in $r=0$ as well.
\end{corollary}
%%%%%%%%%%%%%%%%%%%%%%%%%%%%%%%%%
\begin{proof} We simply recapitulate the arguments and the ideas of the proof of 
Theorem \ref{thm_Stokes_n}.
\end{proof}
%%%%%%%%%%%%%%%%%%%%%%%%%%%%%%%%%
%\begin{corollary} \label{col_thm_sigma_Stokes_n}
\noindent The translation of Theorem \ref{thm_Stokes_n} to the %first
                                %eigenvalues and one of the first
                                %eigenfunctions of
Stokes 
operator ${\boldsymbol S}_{{\sigma}}$ on the ${\Omega}^{*}_{(n),\sigma}$ 
in ${\bb R}^{n}$ is simple. One uses the conversion formulas in
Eq. \eqref{kappa_A_sigma}, namely, 
$\kappa_{1,S}({\sigma})  =({1+{\cal A}/{2}})
\kappa_{1,S}({\cal A})\,$ and the conversion between 
the quantities ${\cal A}$ and ${\sigma}$.
\begin{corollary} \label{col_thm_sigma_Stokes_n}
For any $n\,\in\,{\bb N},\,n>1$ and for any
${\sigma}\,\in\,(0,1)$
% and ${\Omega}_{{\sigma}}={\Omega}_{(n),{\sigma}}\subset {\bb R}^{n}$ 
the first eigenvalues
${\lambda}_{1,S}({\sigma})$ of the Stokes operator on
${\Omega}^{*}_{(n),{{\sigma}}},\,\sigma\,\in\,(0,1)$ are 
the squares of the smallest positive solutions
$\kappa_{1,S}({\sigma})$
of %the following transcendental equations for arbitrary ${\sigma}\,\in\,(0,1)$:
\begin{align}\label{Stokes_Allgem_EW_sigma}
0\,=\,\left\{\begin{array}{ll}
J_{\frac{n}{2}}(\kappa_{S}({\sigma}))Y_{\frac{n}{2}}(\kappa_{S}({\sigma}) {\sigma})
\,-\,J_{\frac{n}{2}}(\kappa_{S}({\sigma}){\sigma})Y_{\frac{n}{2}}(\kappa_{S}({\sigma})) 
& \forall \,
n\mbox{ even} \,\\[2mm]
J_{\frac{n}{2}}(\kappa_{S}({\sigma}))J_{-\frac{n}{2}}(\kappa_{S}({\sigma}) \sigma)
\,-\,J_{\frac{n}{2}}(\kappa_{S}({\sigma}) {\sigma})J_{-\frac{n}{2}}(\kappa_{S}({\sigma})) & \forall \,
n \mbox{ odd} \,.
\end{array}
\right.
\end{align}
The corresponding first eigenfunctions %on
% ${\Omega}^{*}_{(n),{{\sigma}}},\,\sigma\,\in\,(0,1)$
are %the fields
\begin{align}\label{Stokes_Allgem_EF_sigma}
{\underline{w}}_{{\;\!}\varphi,S,{\sigma}}({\underline{x}})\,:=\,{\tilde{c}_{1,S,{\sigma}}}
\left\{
\begin{array}{ll}
\, \varrho(r) \sin{\vartheta_{1}}\sin{\vartheta_{2}}\cdots \sin{\vartheta_{n-2}}\cdot 
{\underline{\mathfrak{e}}}_{{\;\!}\varphi}\,& \forall \,n\,\in\,{\bb N}\,;\,n>2  \\
{~} & {~} \\ 
\, \varrho(r) \cdot 
{\underline{\mathfrak{e}}}_{{\;\!}\varphi}\,& 
\text{for }\,
\,n=2 \, 
\end{array}\right.\,,
\end{align}
(cf. Corollary \ref{Stokes_sph_harm_l=1} and the corresponding spherical harmonics of the degree 
$\ell = 1$ given in Eq. \eqref{SphSurf1}).
Again, $\tilde{c}_{1,S,{\sigma}}$ are  
scaling constants % in the ${\underline{\bb L}}_{{2}}(.)$-sense
and the functions $\varrho(r)$ only depend on $r\,\in\,({\sigma},1)$
and are given as
\begin{align}\label{Stokes_Allgem_Rho_sigma}
\varrho(r)
\,:=\,\frac{1}{r^{\frac{n}{2}-1}}
\left\{
\begin{array}{ll}
(J_{\frac{n}{2}}(\kappa_{1,S}({\sigma})r)\,-\,
\frac{J_{\frac{n}{2}}({\scriptstyle{\kappa_{1,S}({\sigma}){\sigma}}})}
{Y_{{\frac{n}{2}}}({\scriptstyle{\kappa_{1,S}({\sigma}){\sigma}}})}
Y_{\frac{n}{2}}( \kappa_{1,S}({\sigma}) r)
)& \forall \,
n\mbox{ even} \,,\\[2mm]
(J_{\frac{n}{2}}(\kappa_{1,S}({\sigma})r)\,-\,
\frac{J_{{\frac{n}{2}}}({\scriptstyle{\kappa_{1,S}({\sigma}){\sigma}})}}
{J_{-{\frac{n}{2}}}({\scriptstyle{\kappa_{1,S}({\sigma}){\sigma}}})}
J_{-\frac{n}{2}}( \kappa_{1,S}({\sigma}) r)
)&\forall \,
n\mbox{ odd}\, . 
\end{array}
\right.
\end{align}
\end{corollary}
\begin{proof} One repeats the arguments and the ideas from the proof of the
Theorem \ref{thm_Stokes_n} again.
\end{proof}
%%%%%%%%%%%%%
\subsection{First Eigenvalues of the Laplacian with respect to the first Eigenvalues of the Stokes Operator \label{sec_sto_lap_n}}
%%%%%%%%%%%%%
We formulate a shift theorem between the first eigenvalues of the
Laplacian onto the first eigenvalues 
of the Stokes operator, which is the principal result of our investigations.
In what follows we will write the dimension $n$ of the spaces ${\bb R}^{n}$ as an index in the notation
of the roots, of the eigenvalues and of the Poincaré constants too.
\begin{theorem} \label{thm_shift_Lapl_Stokes_n}
Let $n >1$  and consider ${\Omega}_{(n),{\cal
    A}},{\Omega}_{(n)}^{*}\subset {\bb R}^{n}$ resp.  
${\Omega}_{(n+2),{\cal A}},{\Omega}_{(n+2)}^{*}\subset {\bb R}^{n+2}$.
The first eigenvalues ${\lambda}_{1,S,(n)}({\cal A})$  resp. ${\lambda}_{1,S,(n)}(0)$ of the Stokes operator 
on ${\Omega}_{(n),{\cal A}}$ resp. on ${\Omega}_{(n)}^{*}$ in  ${\bb
  R}^{n}$ coincide with the first eigenvalues 
${\lambda}_{1,L,(n+2)}({\cal A})$  resp. ${\lambda}_{1,L,(n+2)}(0)$ of the Laplacian
on ${\Omega}_{(n+2),{\cal A}}$ resp. on ${\Omega}_{(n+2)}^{*}$ in  ${\bb R}^{n+2}$ for any
${\cal A}\,\in\,[0,\infty)$. 
% For any${\cal A}\,\in\,[0,\infty)$
Moreover, the transcendental equations
\eqref{Stokes_Allgem_EW} have the same roots % have the same
% roots
$\kappa_{1,S,(n)}({\cal A})\,=\kappa_{1,L,(n+2)}({\cal A})$ (for $n$) %of the transcendental equations \eqref{Stokes_Allgem_EW} for $n$
as the transcendental equations \eqref{Lapl_Allgem_EW} for $n+2$.
In addition, due to  \eqref{Stokes_einheitskugel_EW} and 
\eqref{Lapl_einheitskugel_EW} we see that \\[.25cm]
\hspace*{5cm}
$J_{\frac{n}{2}}(\kappa_{1,S,(n)})\,=\,J_{\frac{n+2}{2}-1}(\kappa_{1,L,(n+2)})\,=\,0\,.$\\[.1cm]
In particular,  %this implies that
for the Poincar\'{e} constants it holds that $c_{p,S,(n)}({\cal A})\,=\,c_{p,(n+2)}({\cal A})$
$ \forall\,{\cal A}\,\in\,[0,\infty)$ and  $\forall\,n \,\in\,{\bb N} ,\,n >1$.
\end{theorem} 
%%%%%%%%%%%%%%%
\begin{proof} We only have to consider the transcendental equations
  for the first roots  
$\kappa_{1,L,(n+2)}({\cal A})$ and $\kappa_{1,S,(n)}({\cal A})$ 
in Theorems \ref{thm_Lapl_n} and \ref{thm_Stokes_n}.
The equations \eqref{Lapl_Allgem_EW} for $n+2$  and \eqref{Stokes_Allgem_EW}  for $n$ agree
for arbitrary ${\cal A}\,\in\,(0,\infty)$. It is trivial to conclude that  
the smallest positive solutions $\kappa_{1,L,(n+2)}({\cal A})$ and $\kappa_{1,S,(n)}({\cal A})$
coincide. With the same argument we look on Note \ref{Trieb_Laplace_1} for the Laplace 
operator ${\boldsymbol L}_{o}$ on the unit ball ${\Omega}_{(n+2)}^{*}$
and Corollary \ref{Stokes_Unit_ball}
for the Stokes operator ${\boldsymbol S}_{o}$ on the unit ball ${\Omega}_{(n)}^{*}$.
There are the $\kappa_{1,L,(n+2)}(0)$, namely
the smallest positive roots of \eqref{Lapl_einheitskugel_EW} and the
$\kappa_{1,S,(n)}(0)$, which are 
the smallest positive roots of \eqref{Stokes_einheitskugel_EW}.
The simplest formulation is: $\kappa$ is the smallest positive root of
$  
J_{\frac{n}{2}}(\kappa) \,=\,0 \,\,\,\,\text{and }\, \kappa\,=\,\kappa_{1,L,(n+2)}(0)\,=\,\kappa_{1,S,(n)}(0)\,=
\,\kappa_{1,L,(n+2)}(o)\,=\,\kappa_{1,S,(n)}(o)
$.
\end{proof}
%%%%%%%%%%%%%%%%
\noindent The use of the conversion formulas \eqref{kappa_A_sigma}
provide the shift properties on ${\Omega}^{*}_{(n),\sigma}$.
%%%%%%%%%%%%%%%
\begin{remark} \label{rm_shift_Lapl_Stokes_sigma_n}
Let $n >1$.
The first eigenvalue ${\lambda}_{1,S,(n)}({\sigma})$ of the Stokes operator ${\boldsymbol S}_{\sigma}$
on ${\Omega}^{*}_{(n),\sigma}$ in  ${\bb R}^{n}$ coincides with the first eigenvalue
${\lambda}_{1,L,(n+2)}(\sigma)$  of the Laplacian ${\boldsymbol L}_{\sigma}$
on ${\Omega}^{*}_{(n),\sigma}$  in  ${\bb R}^{n+2}$ for all
$\sigma\,\in\,(0,1)$. 
We have the same roots $\kappa_{1,S,(n)}(\sigma)\,=\kappa_{1,L,(n+2)}(\sigma)$ of the 
transcendental equations in Corollary \ref{col_thm_sigma_Stokes_n}\, for $\kappa_{1,S,(n)}(\sigma)$
and of the transcendental equations \eqref{Lapl_sigma_GL_EW} from Corollary \ref{col_sigma_Lapl_n} for 
$\kappa_{1,L,(n+2)}(\sigma)$ 
$\forall \,\sigma\,\in\,(0,1)$.
In particular, this implies
for the Poincaré constants that $c_{p,S,(n)}({\sigma})\,=\,c_{p,(n+2)}({\sigma})$
for all $\sigma\,\in\,(0,1)$ and  for all $n \,\in\,{\bb N} ,\,n >1$.
\end{remark} 
\begin{proof}
One can transmit the proof of Theorem \ref{thm_shift_Lapl_Stokes_n} or simply apply 
the conversion formulas \eqref{kappa_A_sigma}.
\end{proof}
\noindent Using Theorem \ref{thm_shift_Lapl_Stokes_n} it is
possible to transfer a characteristic property
of the first simple eigenvalues ${\lambda}_{1,L,(n+2)}({\cal A})$ onto the first eigenvalues ${\lambda}_{1,S,(n)}({\cal A})$ 
for any ${\cal A}\,\in\,{\blue{[}}(0,\infty)$.
%%%%%%%%%%%%%
\begin{note} \label{small_gap_trick}
For any ${\cal A}\,\in\,[0,\infty)$ the first eigenvalue ${\lambda}_{1,S,(n)}({\cal A})$ of the Stokes operator on ${\Omega}_{(n),{\cal A}}$ 
possesses exactly the same property as the first simple eigenvalue ${\lambda}_{1,L,(n+2)}({\cal A})$
of the Laplace operator on ${\Omega}_{(n+2),{\cal A}}$.
The first eigenvalue ${\lambda}_{1,S,(n)}({\cal A})$ is highlighted by
the special feature that it is
equal to the first simple eigenvalue ${\lambda}_{1,L,(n+2)}({\cal A})$
of the Laplace operator on ${\Omega}_{(n+2),{\cal A}}$ \, for any ${\cal A}\,\in\,[0,\infty)$.
This  could be understood as another tool for the proof, that
${\lambda}_{1,S,(n)}({\cal A})$ is
the first eigenvalue of the Stokes operator on ${\Omega}_{(n),{\cal
    A}}$. One can use this feature to explain  
the behaviour of ${\lambda}_{1,S,(n)}({\cal A})$,  ${\kappa}_{1,S,(n)}({\cal A})$ and $c_{p,S,(n)}({\cal A})$
for ${{\cal A}\,\to\,{\infty}}$.
\end{note} 
%%%%%%%%%%%%
\begin{remark} \label{PiQuadrat1}
We know for the one-dimensional Laplace problem that the first
eigenvalue ${\lambda}_{1}$ satiesfies ${\lambda}_{1}\,=\,{\pi}^2\,$.
This is also the idea behind the small gap limit in  ${\bb R}^{n}, n>1$. 
Namely, for the first simple eigenvalue it holds for ${\cal A}\,\to\,\infty$ that ${\lambda}_{1,L}({\cal A})\,\to\,{\pi}^2$. 
In particular, for all ${\cal A}\,\in\,(0,\infty)$ we use Subsection \ref{Sec2avBou} where in Eq. \eqref{eq:poincarepolar}
%%%%%%%%%
$\tilde w$ is varying in 
$ {\bb
  W}_{2}^{1}\hspace{-.62cm}{~}^{{~}^{{~}^{o}}}\hspace{.2cm}(\frac{\mathcal
  A}{2},1+\frac{\mathcal A}{2})$ to see that  (cf. \cite{nazarov2000})
\begin{equation*}
\label{hallos}
\left(
c_{p}({\cal A})
\right)^2 
=
\left(\frac{1}{\pi}\right)^2 = 
\max_{\tilde w } \frac{
\int_{\frac{\mathcal A}{2}}^{1+\frac{\mathcal A}{2}}
\abs{\tilde w(r)}^2 \, dr}
{
\int_{\frac{\mathcal A}{2}}^{1+\frac{\mathcal A}{2}}
\abs{\tilde w'(r)}^2 \, dr}\,.
\end{equation*}
\end{remark}
%%%%%%%%%%%%%%%%%%%%
\section{Investigation of the Limiting Cases \label{Inv_Lim}}
%%%%%%%%%%%%%%%%%%%%
Subsequently we will study the limiting cases ${\cal A}\,\to\,0$
and ${\cal A}\,\to\,{\infty}$  
separately, because the methods are completely different.
It is quite obvious that the cases ${\cal A}\,\to\,0$ and  ${\sigma}\,\to\,0$ are almost
identical and we formally obtain the punctured ball
\begin{align}
{\Omega}^{*}_{(n)\,\setminus \{{\underline{0}}\}} :=  
{\Omega}^{*}_{(n)\,\setminus \{{\underline{0}}\}} \quad\mbox{in the limits}\quad
 \lim_
{{\sigma}\,\to\,0}{\Omega}^{*}_{(n),\sigma}={\Omega}^{*}_{(n)\,\setminus
  \{{\underline{0}}\}}  \,,\quad 
 \lim_{{\cal A}\,\to\,0}{\Omega}_{(n),\cal A}={\Omega}^{*}_{(n)\,\setminus
  \{{\underline{0}}\}} \,,\quad 
{\Omega}^{*}_{(n)\,\setminus
  \{{\underline{0}}\}}  \simeq {\Omega}_{(n)}^{*}\,.
\end{align}
For simplicity we will choose the parameter ${\sigma}$ and may use
(\ref{calA}) to connect ${\cal A}$ and 
${\sigma}$ for the (first) eigenvalues. So we note, e.g., for the
smallest eigenvalue (the square of the smallest positive zero) by the  conversion formulas  \eqref{kappa_A_sigma} that: 
\begin{align}
\lambda_{1,L}({\sigma})\,=\,
(\kappa_{1,L}({\sigma}))^2\,=\,(({1+{\cal A}/{2}}) \kappa_{1,L}({\cal A}))^2\,=\,({1+{\cal A}/{2}})^2 \lambda_{1,L}({\cal A})\,.
\end{align}
% In our study of ${\cal A}\,\to\,{\infty}$
%The parameter ${\cal A}$ is replaced by the radius $R\,={\cal A}/2$
%for convenience.
Looking for the 
limit of 
${\Omega}_{(n),{\cal A}}$ as ${\cal A}\to \infty$ one easily sees, that
they are ``losing their curvature'' and in the limit become an
unbounded layer also known as the {\em small gap limit}
\[
{\Omega}_{(n),{\cal A}=\infty}\,:=
\{{(s,\tau_1,\dots,\tau_{n-1})^{T}} \in {\bb R}^{n}:\,0 < s < 1\}\,.
\]
It is worth to note,
that in the limiting process ${\cal A}\,\to\,{\infty}$ (or $R\,={\cal
  A}/2\to\infty$) the Laplace and 
Stokes operator lose the property of being operators with a pure point spectrum. 
But nevertheless, they stay linear positive operators and this ensures
the existence of smallest eigenvalues. 
%In contrast to our ideas for coordinate transformations in
%\cite{RuRuTh2016} we are going to employ considerations of the
%limiting value.
We will use 
the transcendental equations \eqref{Lapl_Allgem_EW} and \eqref{Stokes_Allgem_EW} 
for the square root of the eigenvalues 
in the limiting process at this point for
${\cal A}\,\to\,{\infty}$ (or $R\,={\cal A}/2\,\to\,{\infty}$). 
\subsection{The behaviour of the Laplace eigenvalues for
  $\boldsymbol{\sigma\,\to\,0}$  
\label{Green} }
We use the notations ${\boldsymbol L}_{\sigma}$ for the  Laplace operator on
${\Omega}^{*}_{(n),\sigma}$ and ${\boldsymbol L}_{\sigma}^{-1}$ for its inverse at ${\sigma}\,\in\,(0,1)$.
The properties of the first simple eigenvalue and the first eigenfunctions of ${\boldsymbol L}_{\sigma}$ 
on ${\Omega}^{*}_{(n),\sigma}$ are a simple Corollary of Theorem
\ref{thm_Lapl_n} using the conversion formula above.
%%%%%%%%%%%%%%%%%%%%%%%%%%%%%%%%%%%%%%%%%%%%%%%%%
\begin{corollary} \label{col_sigma_Lapl_n} %Let $n>3$ {\blue{(resp. $n>1$)}} and ${\Omega}^{*}_{(n),\sigma}\subset {\bb R}^{n}$.
For all $n\,\in\,{\bb N},\,n>3,
$ and any $\sigma\,\in\,(0,1)$ the first simple eigenvalue
${\lambda}_{1,L}({\sigma})$ of the Laplacian ${\boldsymbol L}_{\sigma}$  on
${\Omega}^{*}_{(n),\sigma}$
is the square of the smallest positive solutions
$\kappa_{1,L}({\sigma})$ 
of the following transcendental equations:  
\begin{align}\label{Lapl_sigma_GL_EW}
0\,=\,\left\{\begin{array}{ll}
J_{\frac{n}{2}-1}(\kappa_{L}({\sigma}))Y_{\frac{n}{2}-1}(\kappa_{L}({\sigma}) \sigma)
\,-\,J_{\frac{n}{2}-1}(\kappa_{L}({\sigma}) \sigma)Y_{\frac{n}{2}-1}(\kappa_{L}({\sigma})) 
& \forall \,
n\mbox{ even} \,\\[2mm]
%{~} & {~} \\
J_{\frac{n}{2}-1}(\kappa_{L}({\sigma}))J_{-\frac{n}{2}+1}(\kappa_{L}({\sigma}) \sigma)
\,-\,J_{\frac{n}{2}-1}(\kappa_{L}({\sigma}) \sigma)J_{-\frac{n}{2}+1}(\kappa_{L}({\sigma})) & \forall \,
n\mbox{ odd}\,.
\end{array}
\right.
\end{align}
The corresponding first eigenfunction is
\begin{align}\label{Lapl__sigma_Allgem_EF}
w_{{\;\!}1,L,\sigma}({\underline{x}})
\,=\,\frac{\tilde{c}_{1,L,\sigma}}{r^{\frac{n}{2}-1}}
\cdot
\left\{
\begin{array}{ll}
(
J_{\frac{n}{2}-1}(\kappa_{1,L}(\sigma)r)\,-\,
\frac{J_{{\frac{n}{2}}-1}({\textstyle{{\kappa_{1,L}(\sigma) \sigma}}})}
{Y_{{\frac{n}{2}}-1}({\textstyle{{ \kappa_{1,L}(\sigma) \sigma}}})}
Y_{\frac{n}{2}-1}( \kappa_{1,L}(\sigma) r
)
)& \forall \,
n\mbox{ even} \,,\\[2mm]
(J_{\frac{n}{2}-1}(\kappa_{1,L}(\sigma)r)\,-\,
\frac{J_{{\frac{n}{2}}-1}({\textstyle{{\kappa_{1,L}(\sigma) \sigma}}})}
{J_{-{\frac{n}{2}}+1}({\textstyle{\kappa_{1,L}(\sigma) \sigma}})}
J_{-\frac{n}{2}+1}( \kappa_{1,L}(\sigma) r
)
)&\forall \,
n\mbox{ odd} \,,
\end{array}
\right.
\end{align}
where again, 
$\tilde{c}_{1,L,\sigma}$ denotes the
scaling constants. %  in the ${{\bb L}}_{{2}}(.)$-sense.
The functions in \eqref{Lapl__sigma_Allgem_EF} only 
depend on $r\,\in\,(\sigma,1)$ (see Thm.~\ref{thm_Lapl_n}).
\end{corollary}
\noindent Additionally we denote the Laplace 
operator on the unit ball
${\Omega}^{*}_{(n)}$ by ${\boldsymbol L}_{o}$ and its inverse 
${\boldsymbol L}_{o}^{-1}$, respectively.
As in \cite{RuRuTh2016} we use the Green's function to study the limit
$\sigma\to0$. The case $n=2$ is treated in
\cite[Subsec.3.1]{RuRuTh2016} and thus, we restrict ourselves to the
case $n\ge 3$, $n\in{\bb N}$.
% For $n=2$ are the Green's
                                %functions given in \cite{RuRuTh2016} cf. subsection 3.1. We make the commitment that we regard now the dimensions $n\geq 3,\, n\in{\bb N}$.\\ 
First, we us use the Green's functions for the
(negative) Laplacians to define the inverse operators. For this
purpose let 
${\underline{x}}$ and ${\underline{y}}$  be two
points in 
${\overline{{\Omega}}}_{(n)}^{{\;\!}*}$. 
Then Green's function for the
(negative) Laplacian on ${\Omega}^{*}_{(n)}$ with zero trace at fixed $n\geq 3$ is given by (cf.eg. \cite{CouHil}):
\begin{equation*}
G_{(n)}({\underline{x}},{\underline{y}})=
\frac{1}{(n-2){|\omega}_{(n)}|}\left(\frac{\displaystyle{1}}{{\displaystyle{
\|{\underline{x}}-{\underline{y}}\|}}^{n-2}}\,-\,
\frac{\displaystyle{1}} 
{\displaystyle{(\|{\underline{x}}\|^{2}
\|{\underline{y}}\|^{2}+1-2\|{\underline{x}}\|
\|{\underline{y}}\|\cos({\underline{x}},{\underline{y}}))^{\frac{n-2}{2}}}}\,\right).
\end{equation*} 
The symmetric form of the Green's function for the (negative)
Laplacian with vanishing traces on $\partial{\Omega}^{*}_{(n),\sigma}=
\omega_{_{(n)},\sigma} \cup \omega_{(n)}$ on 
every ${\Omega}_{(n),\sigma}^{*}$ (for fixed $\sigma\,>\,0$) is (cf. \cite{GilTru})
\begin{eqnarray}\label{sternstern}
  % &&\hspace*{-16mm}
       G_{(n),\sigma}({\underline{x}},{\underline{y}})&=&\frac{1}{(n-2)|{\omega}_{(n)}| } \sum\limits_{k\in\mathbb{Z}}
\bigg( \frac{\sigma^{k(n-2)}}{ \| {\underline{y}}-\sigma^{2k}{\underline{x}} \|^{n-2}}
-\frac{\sigma^{k(n-2)}}{\left\|\,  \| {\underline{x}} \|  {\underline{y}} -\frac{{\displaystyle{\sigma^{2k}}}}
{{\displaystyle{ \| {\underline{x}} \|}}} {\underline{x}} 
\right\|^{n-2}}
  \bigg) %\,G_{(n)}({\underline{x}},{\underline{y}})\,+\quad
  % \nonumber
  \\
&=& G_{(n)}({\underline{x}},{\underline{y}})\,+\frac{1}{(n-2)|{\omega}_{(n)}| }\sum_{k\in {\bb N}}{{\sigma}^{k(n-2)}}\bigg(\frac{1}{\displaystyle{\|{\underline{y}}-{\sigma}^{2k}{\underline{x}}\|^{n-2}}}
+\frac{1}{\displaystyle{\|{\sigma}^{2k}{\underline{y}}-{\underline{x}}\|^{n-2}}}+
\hspace*{2cm}\quad\quad\nonumber\\ 
&&\frac{-1}
{\displaystyle{(\|{\underline{x}}\|^{2}
\|{\underline{y}}\|^{2}\!\!+{\sigma}^{4k}\!\!-\!\! 2{\sigma}^{2k}\|{\underline{x}}\|
\|{\underline{y}}\| 
\cos({\underline{x}},{\underline{y}}))^{\frac{n}{2}-1}}}+\frac{-1}
{\displaystyle{({\sigma}^{4k}\|{\underline{x}}\|^{2}
\|{\underline{y}}\|^{2}\!\!+\!\! 1 \!\!-\!\! 2{\sigma}^{2k}\|{\underline{x}}\|
\|{\underline{y}}\| 
\cos({\underline{x}},{\underline{y}}))^{\frac{n}{2}-1}}}\bigg)\,\,,
   \nonumber
\end{eqnarray}
where the points ${\underline{x}}$ and ${\underline{y}}$ have to be in
${\overline{{\Omega}}}^{{\;\!}*}_{(n),\sigma}$. This shape of the
Green's functions is also useful for the study 
of ${\sigma}\,\to\,0$. By a straightforward calculation we get for the
limits for any $n\,\in\,{\bb N}:\, n\geq 3$ that
\begin{align}
\label{Lapl_GWsig}
\lim_{{\sigma}\,\to\,0}{\;\!}
  G_{(n),\sigma}({\underline{x}},{\underline{y}})=G_{(n)}({\underline{x}},{\underline{y}})
  \, .
\end{align}
%where the simple calculations and arguments are more or less the same as in \cite{RuRuTh2025}.
The basic properties of the eigenfunctions of  ${\boldsymbol L}_{o}$  
and Bessel's differential operator can be found in \cite[Chapters 5 and 8]{Triebel} together with the explicit 
representation of the inverse Laplacian by  Fourier series in 
the eigenfunctions. \\[.1cm]
The first eigenfunctions of the Laplacian ${\boldsymbol L}_{o}$
on the unit ball
${\Omega}_{(n)}^{*}\subset {\bb R}^{n}$ are written in formula
\eqref{Lapl_einheitskugel_EF} as
\begin{align*}
w_{{\;\!}1,L,o}({\underline{x}})\,=\,
{{w}}_{{\;\!}1,L}({\underline{x}})
\,=\,\frac{\tilde{c}_{1,L}}{r^{\frac{n}{2}-1}}
\cdot
J_{\frac{n}{2}-1}(\kappa_{1,L,(n)}({o})\cdot r)\,,
\end{align*}
where the constants
$\tilde{c}_{1,L}$ are the scaling constants. % in the ${{\bb L}}_{{2}}(.)$-sense.
The corresponding first simple eigenvalues $\lambda_{1,L,(n)}({o})$ are the squares of the smallest 
positive roots $\kappa_{L}\,=\,\kappa_{1,L,(n)}({o})$ of
\begin{align*}
J_{\frac{n}{2}-1}(\kappa_{L})=0 \quad \quad \quad \quad \,\forall \,n\,\in\,{\bb N}:\,n>1\,\,\,:\,\,\,\lambda_{1,L,(n)}({o})\,=\,(\kappa_{1,L,(n)}({o}))^{2}\,\,.
\end{align*}
Numerical computations show that we %have in ascending order of the
% dimension
for $n\,= \,2, \,3 ,\,4 ,\,5,\,6,\,7\dots$
by way of example  $\lambda_{1,L,(2)}({o})\,\approx\,(2.404825558)^2$,\,\, 
$\lambda_{1,L,(3)}({o})\,=\,\pi^2$,\,\,$\lambda_{1,L,(4)} ({o})\,\approx\,(3.831705970)^2$,\,\,
$\lambda_{1,L,(5)}({o})\,\approx\,(4.493409458)^2$,\,\,$\lambda_{1,L,(6)}({o})\,\approx\,(5.135622302)^2$,\\
$\lambda_{1,L,(7)}({o})\,\approx\,(5.763459197)^2$,\,\,$\dots$\,\,.\\
The operators ${\boldsymbol L}_{o}^{-1}$ and ${\boldsymbol L}_{\sigma}^{-1}$ (cf. Corollary  \ref{col_sigma_Lapl_n} for the 
Laplacian ${\boldsymbol L}_{\sigma}$ on ${\Omega}^{*}_{(n),\sigma}\subset {\bb R}^{n}$)
are selfadjoint, positive and 
compact. Thus, we observe
\begin{align}
\label{neu_n_Lapl}
     (\lambda_{1,L,(n)}({o}))^{-1}&= \max_{u \in {\bb
        L}_{2}({\Omega}_{(n)}^{*}): \| u \|_{{\bb L}_{2}(.)}=1 }{\;\!}
    \int_ {{\Omega}^{*}_{(n){~}}} \int_ {{\Omega}^{*}_{(n){~}}}
    G_{(n)}({\underline{x}},{\underline{y}})u({\underline{y}})
    u({\underline{x}})d{\underline{y}}d{\underline{x}}=\|{\bf
      L}_{o}^{-1}\|
    % \quad\quad{\mbox{and}}
    \\
 (\lambda_{1,L,(n)}({\sigma}))^{-1}&=
    \max_{u \in {\bb L}_{2}({\Omega}^{*}_{(n),\sigma}): \| u \|_{{\bb
          L}_{2}(.)}=1 }{\;\!}  \int_ {{\Omega}^{*}_{(n),\sigma}} \int_
    {{\Omega}^{*}_{(n),\sigma}}
    G_{(n),\sigma}({\underline{x}},{\underline{y}})
    u({\underline{y}})u({\underline{x}})d{\underline{y}}d{\underline{x}}=\|{\bf
      L}_{\sigma}^{-1}\|\,\nonumber,
\end{align}
where the maxima are attained at $u=w_{1,L,o}({\underline{x}})$ and at
$u=w_{1,L,{\sigma}}({\underline{x}})$, 
respectively. The equations above immediately show the behaviour of the Laplace eigenvalues for
${\sigma\,\to\,0}$ too. We also apply an important classical 
tool for partial differential equations \cite{CouHil}. %We use the
                                %notations
                                %$(\lambda_{1,L,(n)}({o}))^{-1}=\|{\bf
                                %L}_{o}^{-1}\|\quad \mbox{and}\quad (\lambda_{1,L,(n)}({\sigma}))^{-1}=\|{\bf L}_{\sigma}^{-1}\|\, $ again to describe in detail the central ideas.
For ${\sigma}\in\,(0,1)$ and $ {\sigma}_{o}\geq {\sigma}$ it is obvious, that
${\Omega}_{(n),{\sigma}_{o}}^{*}\subset {\Omega}_{(n),{\sigma}}^{*}\subset 
{\Omega}^{*}_{(n)}$. Thus, the values
$(\lambda_{1,L,(n)}({\sigma}))^{-1}$ constitute a continuous, 
monotonically decreasing function 
with respect to ${\sigma}$ (applying
\cite[Satz 3 on page 355 in Bd. 1]{CouHil}) for every fixed
$n\,\in\,{\bb N}:\,n>1$ and \cite{Weid}.	
In addition we see 
for any fixed ${\sigma}_{o}<1$ that
\begin{equation}\label{stern}
(\lambda_{1,L,(n)}(\sigma_{o}))^{-1} \leq  (\lambda_{1,L,(n)}({\sigma}))^{-1}
\leq (\lambda_{1,L,(n)}({o}))^{-1}
\quad
\mbox{for}\quad {\sigma}\in\,(0,{\sigma}_{o})\,. 
\end{equation}
We use the {\em constrained
  subsets} of ${\bb L}_{2}({\Omega}^{*}_{(n)})$ as a second tool for
our proof. 
Let us introduce the space
\[
{\bb L}_{2}({\Omega}^{*}_{{(n)},\backslash {\sigma}}):=\{v\,\in\,{\bb
  L}_{2}({\Omega}^{*}_{(n)}): v=0\,\, {\mbox{a.e. on }} {\Omega}^{*}_{(n)}
\setminus {\Omega}^{*}_{(n),\sigma}\}
\]
of almost everywhere vanishing functions for 
${\underline{x}} \in {\bb R}^{n}$ with $\|{\underline{x}}\|\,<{\sigma}$. It
is obvious, that 
${\bb L}_{2}({\Omega}^{*}_{(n),\backslash {\sigma}} $$)\,
\simeq\, {\bb
  L}_{2}({\Omega}^{*}_{(n),\sigma}$$)$ 
in the sense of an identity map.
We apply the standard property of maxima on subsets combined with the
central tool of the limiting process  
(\ref{Lapl_GWsig}) to see, that there exists 
\[ 
\lim_{{\sigma}\,\to\,0}{\;\!}  (\lambda_{1,L,(n)}({\sigma}))^{-1}
\leq  (\lambda_{1,L,(n)}({o}))^{-1}
\]
and that $\lambda_{1,L,(n)}({\sigma})$ is continuous in $0$ with the value of the function $\lambda_{1,L,(n)}({0})\,=\,
\lambda_{1,L,(n)}({o})$
(see \cite{Weid}). 
Using the definitions of $G_{(n)} $
and $G_{(n),\sigma}$ and that for every $u\,\in$
${\bb L}_{2}({\Omega}^{*}_{(n)\,\setminus \{{\underline{0}}\}})$ we also
have $u \in {\bb L}_{2}({\Omega}^{*}_{(n)} )$ we conclude
\begin{align}\label{neu1}
\int_{ {\Omega}^{*}_{(n){~}}}\int_ {{\Omega}^{*}_{(n){~}}} G_{(n)}({\underline{x}},{\underline{y}})
u({\underline{y}})u({\underline{x}})d{\underline{y}}d{\underline{x}}=
\int_{{\Omega}^{*}_{(n)\,\setminus \{{\underline{0}}\}}}\int_{{\Omega}^{*}_{(n)\,\setminus \{{\underline{0}}\}}}
\lim_{{\sigma}\,\to\,0} G_{(n),\sigma}({\underline{x}},{\underline{y}}) 
u({\underline{y}})u({\underline{x}})d{\underline{y}}d{\underline{x}}\quad\forall
  u \in {\bb L}_{2}({\Omega}^{*}_{(n)} )\,\,\,. 
\end{align}
From \eqref{neu_n_Lapl}, \eqref{neu1} and the continuity of $\lambda_{1,L,(n)}({\sigma})$ in $\sigma = 0 (=o)$ 
we may conclude that 
\begin{align*}
&\max_{u \in {\bb L}_{2}({\Omega}^{*}_{{(n)}}):   \| u \|_{{\bb L}_{2}(.)}=1  }{\;\!}
\int_{{\Omega}^{*}_{(n)\,\setminus \{{\underline{0}}\}}}\int_{{\Omega}^{*}_{(n)\,\setminus \{{\underline{0}}\}}}
\lim_{{\sigma}\,\to\,0} G_{(n), \sigma}({\underline{x}},{\underline{y}}))
u({\underline{y}})u({\underline{x}})d{\underline{y}}d{\underline{x}}
                           \\ 
=&\max_{u \in {\bb L}_{2}({\Omega}^{*}_{(n)}):   \| u \|_{{\bb L}_{2}(.)}=1  }{\;\!}
\int_{ {\Omega}^{*}_{(n)}}\int_ {{\Omega}^{*}_{(n)}} G_{(n)}({\underline{x}},{\underline{y}})
u({\underline{y}})u({\underline{x}})d{\underline{y}}d{\underline{x}}\quad
=\,(\lambda_{1,L,(n)}({0}))^{-1}\,=\,(\lambda_{1,L,(n)}({o}))^{-1}\,=\,\\
  =&\,\,\lim_{{\sigma}\,\to\,0} (\lambda_{1,L,(n)}({\sigma}))^{-1}=
\lim_{{\sigma}\,\to\,0}
\Big( \max_{u \in {\bb L}_{2}({\Omega}^{*}_{(n),\sigma}):   \| u \|_{{\bb L}_{2}(.)}=1  }{\;\!}
\int_ {{\Omega}^{*}_{(n),\sigma}}   \int_ {{\Omega}^{*}_{(n),\sigma}}  (
G_{(n),\sigma}({\underline{x}},{\underline{y}}))
u({\underline{y}})u({\underline{x}})d{\underline{y}}d{\underline{x}}
\Big).
\end{align*}
%For the third line we used the relations \eqref{sternstern},
%\eqref{stern}
%and monotone convergence.\label{Lapl__sigma_Allgem_EF}
We note, that also the limit in the eigenfunctions 
\begin{align*}
w_{1,L,o}({\underline{x}})=\lim_{{\sigma}\,\to\,0}{\;\!}w_{1,L,{\sigma}}({\underline{x}})
\end{align*}
is well-defined in the sense of almost uniform convergence. Here
$w_{{\;\!}1,L,\sigma}({\underline{x}})$ is regarded as the continuous 
extension of $w_{{\;\!}1,L,\sigma}({\underline{x}})$ on
${\overline{{\Omega}}}^{{\;\!}*}_{(n)}$ by
zero in $\,{\overline{{\Omega}}}^{{\;\!}*}_{(n)}
\setminus{{{\Omega}}}^{{\;\!}*}_{(n),\sigma} $. 
An essential aspect of our study is, that the behaviour of the Green's functions for the
(negative) Laplacians in the process ${\sigma\,\to\,0}$ may be seen as 
forgetting the point ${\underline{0}}$ (with its zero boundary
conditions) in the limit.   
%%%%%%%%%%%%
\subsection{The behaviour of the first Stokes eigenvalues for
  $\boldsymbol{\sigma\,\to\,0}$ \label{StoJ1Y1} } 
%%%%%%%% 
We denote the  Stokes operators on
 ${\Omega}^{*}_{(n),\sigma}$ by ${\bf S}_{\sigma}$ (cf. Remark \ref{rm_shift_Lapl_Stokes_sigma_n})
and the  one
on the unit ball ${\Omega}^{*}_{(n)}$  by ${\bf S}_{o}$ (as in Subsection~\ref{sec_sto_n}) 
$\forall\, n\,\in\,{\bb N} : n >1 $.
The elements of the family of Stokes operators
$\{{\bf S}_{\sigma}\}_{\sigma\,\in\,(0,1)}$ are selfadjoint and positive just as ${\bf S}_{o}$.\\
For the first eigenvalues $\lambda_{1,S,(n)}({o})$ and
$\lambda_{1,S,(n)}({\sigma})$ (counted in the respective fix
multiplicity for ${\sigma}\in\,[0,1)$  and  $n >1 $) we find that:
 \begin{align}
\label{Stok_Eig}
\lambda_{1,S,{(n)}}({o})=
\min_{{\underline{u}} \in {\underline{\bb S}}_{{\;\!}}^{2}({\Omega}^{*}_{(n)}):
\| {\underline{u}} \|_{{\underline{\bb H}}_{{\;\!}}(.)=1  }} ({\bf
   S}_{o}{\underline{u}},{\underline{u}})_{{\underline{\bb H}}_{{\;\!}}(.) }
\quad 
{\mbox{and}}\quad%\nonumber\\
\lambda_{1,S,{(n)}}({\sigma})=
\min_{{\underline{u}} \in {\underline{\bb S}}_{{\;\!}}^{2}({\Omega}^{*}_{{(n)},{\sigma}}):
\| {\underline{u}} \|_{{\underline{\bb H}}_{{\;\!}}(.)=1  }} 
({\bf   S}_{{\sigma}} {\underline{u}},{\underline{u}})_{{\underline{\bb H}}_{{\;\!}}(.) }
\,, 
\end{align}
where the minima are attained at ${\underline{u}}\,=\,{\underline{w}}_{{\;\!}\varphi,S,{\sigma}}$
(cf. Eq. 
\eqref{Stokes_Allgem_EF_sigma}), 
respectively, in ${\underline{u}}\,=\,{\underline{w}}_{{\;\!}\varphi,S}$ 
(cf. Eq. \eqref{Stokes_Allgem_Einsball_EF}).
The eigenvalues $\lambda_{1,S,(n)}({\sigma})$ constitute a continuous, monotonically
increasing function of ${\sigma}\in\,(0,1)$
for any fixed 
$n>1\,;n \in {\bb N}$. This is due to \cite[Satz 3 on p.~355 in Bd.~1]{CouHil}
and the inclusions  
${\Omega}_{{(n)},{\sigma}_{o}}^{*}\subset {\Omega}_{{(n)},{\sigma}}^{*}\subset 
{\Omega}^{*}_{(n)}$ for any $ {\sigma}_{o}\geq {\sigma}$ again
together with and \cite{Weid}.
%%%%%%%%%%%%%%%%%%%%%%%%%%%%%%%%%%%%%%%%%%%%%%%%%%%%%
We observe 
$\lambda_{1,S,(n)}({\sigma}_{o})\geq  \lambda_{1,S,(n)}({\sigma}) \geq
\lambda_{1,S,(n)}({o})$ like for the Laplacian 
at ${\sigma}\in\,(0,{\sigma}_{o})$, for any fixed ${\sigma}_{o}<1$. 
With similar arguments as in Subsection \ref{Green} we
see that the right-hand limit  
of the first Stokes eigenvalues attains the unique value
 \begin{align}\label{eWSlimsigma_0}
  (\kappa_{1,S,(n)}({o}))^{2}\,:= \,
 {\lambda_{1,S,(n)}({o})}=\lim_{{\sigma}\,\to\,0}{\;\!}
   {\lambda_{1,S,(n)}({\sigma})}\,:= \, 
\lim_{{\sigma}\,\to\,0}{\;\!}  \big( {\kappa_{1,S,(n)}({\sigma})}\big)^{2}\,.
 \end{align}
 There one achieves this result by Eq. (\ref{Stok_Eig}).
 \begin{remark} One shows Eq. \eqref{eWSlimsigma_0} by applying
   the shift properties  
 from Theorem \ref{thm_shift_Lapl_Stokes_n}
 and from Remark 
 \ref{rm_shift_Lapl_Stokes_sigma_n} as well.
 \end{remark} 
 %%%%%%%%%%%%%
 % \begin{remark} The use of vector fields in 
%  the Cartesian coordinate system ${\underline{u}}\,=\,{\underline{u}}_{\mathfrak{c}}\in\,{\underline{\bb S}}_{{\;\!}}^{2}(({\Omega}^{*}_{{(n)},{\sigma}})$
% resp.${\underline{u}}\,=\,{\underline{u}}_{\mathfrak{c}}\in\,
% {\underline{\bb S}}_{{\;\!}}^{2}(({\Omega}^{*}_{{(n)}})$ in \eqref{Stok_Eig} result in components of the Laplacian form for 
% $({\bf   S}_{{\sigma}} {\underline{u}},{\underline{u}})_{{\underline{\bb H}}_{{\;\!}}(.) }$ resp.
% $({\bf   S}_{o} {\underline{u}},{\underline{u}})_{{\underline{\bb H}}_{{\;\!}}(.) }$, where
% there are permitted in the variation only vector functions which possess surface harmonic 
% functions at least of a degree $\ell \geq 1 $ in respect to
%  Corollary \ref{Stokes_sph_harm_l=1}.
% \end{remark}
\noindent Like in Subsection~\ref{Green} we use {\em constrained subsets} of 
${\underline{\bb H}}_{{\;\!}}({\Omega}^{*}_{(n)})$ as an appropriate
tool for the limit at ${\sigma}\,\to\,0$.
We define the space of almost everywhere vanishing vector functions in 
${\underline{x}} \in {\bb R}^{n}:\|{\underline{x}}\|\,<{\sigma}$: 
\[{\underline{\bb H}}_{{\;\!}}({\Omega}^{*}_{{(n)},\backslash {\sigma}}):=\{\underline{v}\,\in\,{\underline{\bb H}}_{{\;\!}}({\Omega}^{*}): \underline{v}=\underline{0}\,\,
{\mbox{a.e. on }} {\Omega}^{*}_{(n)}
\setminus {\Omega}^{*}_{(n),\sigma}\}\,.
\]
We see that 
${\underline{\bb H}}_{{\;\!}}({\Omega}^{*}_{{(n)},\backslash {\sigma}})\simeq {\underline{\bb H}}_{{\;\!}}({\Omega}^{*}_{(n),\sigma})$ 
in the sense of an identity map (cf. Subsection~\ref{Green} too).
The space ${\underline{\bb
    H}}_{{\;\!}}({\Omega}^{*}_{(n),\setminus \{{\underline{0}}\}})$ is additionally
equivalent to 
${\underline{\bb H}}_{{\;\!}}({\Omega}^{*}_{(n)})$ , where
${\Omega}^{*}_{(n),\setminus \{{\underline{0}}\}}$ was defined in
Section~\ref{Inv_Lim}.\\
In view of these results now we are in the same situation as in
\cite{RuRuTh2016,RuRuTh2025},  where the cases $n=2$ and $n=3$ were
treated.
Thus, we may conclude with the same methods and estimates as in \cite{RuRuTh2016} and
\cite{RuRuTh2025} that 
$\lim_{{\sigma}\,\to\,0}{\;\!}
{\underline{w}}_{{\;\!}\varphi,S,{\sigma}}({\underline{x}})=
{\underline{w}}_{{\;\!}\varphi,S,o}({\underline{x}})\,=\,{\underline{w}}_{{\;\!}\varphi,S}({\underline{x}})$ 
in the sense of pointwise convergence on
${\overline{{\Omega}}}^{{\;\!}*}_{(n)}$ and in the sense 
of almost uniform convergence too due to the obvious right-hand
continuity of $\tilde{c}_{1,S,{\sigma}}$.
% Now, we have made all the preparations to refer our results in \cite{RuRuTh2016} and
% \cite{RuRuTh2025}. 
% We focus our investigations on the limit ${{\sigma}\,\to\,0}$ again
% regarding the functions ${\underline{w}}_{{\;\!}\varphi,S,{\sigma}}(\underline{x})$
% (cf. Corollary \ref{col_thm_sigma_Stokes_n}, 
% \eqref{Stokes_Allgem_EF_sigma})
% and ${\underline{w}}_{{\;\!}\varphi,S}(\underline{x})$ 
% (cf. Corollary  \ref{Stokes_Unit_ball}, \eqref{Stokes_Allgem_Einsball_EF}).
The functions
${\underline{w}}_{{\;\!}\varphi,S,{\sigma}}(\underline{x})$ are also
considered as its  
continuous extension onto
${\overline{{\Omega}}}^{{\;\!}*}_{(n)}$ by  $\underline{0}$
(the zero vector) in 
${\overline{{\Omega}}}^{{\;\!}*}_{(n)}
\setminus{{{\Omega}}}^{{\;\!}*}_{{(n)},\sigma} $ 
if necessary.\\ 

%%%%%%
\subsection{The behaviour of the first Laplace and Stokes 
eigenvalues for 
$\boldsymbol{{\cal A}\,\to\,{\infty}}$ \label{LaStoInf} } 
%%%%%%%%%%%
We investigate the transcendental equations for the first eigenvalues 
of the Laplace and the
Stokes operators as ${\cal A}\,\to\,{\infty}$ directly like in \cite{RuRuTh2025}. 
The successful use of a coordinate transformation like in \cite{RuRuTh2016} is unfortunately
resticted to $n=2$. But the transformation  
\begin{align} \label{R_und_A}
r\,:=\,R+s \quad \text{for} \quad s\,\in\,[0,1] \quad,\,\, \text{with}\quad
R\,:=\,\frac{\mathcal A}{2}\,
\end{align}
is advantageous for our studies.
The so-called
small-gap limit is more or less
to establish a relationship between the first eigenvalues 
of the Laplace operator (resp. the
Stokes operator) and the one-dimensional Laplace eigenvalue problem as ${\cal A}\,\to\,{\infty}$ 
\cite[Prop.~1.1]{nazarov2000}, where 
$\tilde w$ varies in $ {\bb
  W}_{2}^{1}\hspace{-.62cm}{~}^{{~}^{{~}^{o}}}\hspace{.2cm}(\frac{\mathcal
  A}{2},1+\frac{\mathcal A}{2})
{{ \,=\,{\bb
  W}_{2}^{1}\hspace{-.62cm}{~}^{{~}^{{~}^{o}}}\hspace{.2cm}(R,R+1)}}\,$ and 
\begin{equation}
\label{eq:1DEigenwert}
\hspace*{1.2cm}
{\pi}^2 = 
\min_{\tilde w } \frac{
\int_{\frac{\mathcal A}{2}}^{1+\frac{\mathcal A}{2}}
\abs{\tilde w'(r)}^2 \, dr}
{
\int_{\frac{\mathcal A}{2}}^{1+\frac{\mathcal A}{2}}
\abs{\tilde w(r)}^2 \, dr} {{\,=\,
\min_{\tilde w } \frac{
\int_{R}^{R+1}
\abs{\tilde w'(r)}^2 \, dr}
{
\int_{R}^{R+1}
\abs{\tilde w(r)}^2 \, dr}}}
\,.
\end{equation}
%We are going to bring out the importance of the number ${\pi}^2$ in the following
\begin{prop}%28
\label{PiQuadratinf}
For the first simple eigenvalue ${\lambda}_{1,L,(n)}({\cal A})$ of the 
Laplacian ${\boldsymbol L}$ and for the first eigenvalue 
${\lambda}_{1,S,(n)}({\cal A})$ of the Stokes operator ${\boldsymbol
  S}$ we obtain, that\,
%(cf. Definition \ref{D4})
$\lim_{{\cal A}\to\infty} {\lambda}_{1,L,(n)}({\cal A})\,=\,{\pi}^2$ and \,$\lim_{{\cal A}\to\infty} {\lambda}_{1,S,(n)}({\cal A})\,=\,{\pi}^2$ .
\end{prop}
%%%
% \input{Hallo_kappa.tex}
\begin{proof} {{
The first simple eigenvalue ${\lambda}_{1,L,(n)}({\cal A})$ of the 
Laplace operator ${\boldsymbol L}$ and the first eigenvalue 
${\lambda}_{1,S,(n)}({\cal A})$ of the Stokes operator ${\boldsymbol
  S}$ are the squares of the smallest positive solution 
$\kappa_{1,L,(n)}({\cal A})$ resp. $\kappa_{1,S,(n)}({\cal A})$
(cf. Theorem \ref{thm_Lapl_n} and  Theorem \ref{thm_Stokes_n}) of the
transcendental equations \eqref{Lapl_Allgem_EW} and
\eqref{Stokes_Allgem_EW}, respectively, 
for arbitrary ${\cal A}\,\to\,\infty$, resp. $R\,\to\,\infty$.
Due to the shift Theorem \ref{thm_shift_Lapl_Stokes_n} 
we have the same first roots $\kappa_{1,S,(n)}({\cal
  A})\,=\kappa_{1,L,(n+2)}({\cal A})$ %of the transcendental equations
for all $n>1$. We introduce the abbreviation $\kappa\,:=\,\kappa_{1,L,(n)}({\cal A})$
and with $R\,=\,\frac{\cal A}{2}\,\in\,(0,\infty)$  write 
the equations \eqref{Lapl_Allgem_EW} for $\kappa$ as:
\begin{align}\label{Lapl_Allgem_EW_R}
0\,=\,\left\{\begin{array}{ll}
J_{\frac{n}{2}-1}(\kappa(1+R))Y_{\frac{n}{2}-1}(\kappa R)
\,-\,J_{\frac{n}{2}-1}(\kappa R)Y_{\frac{n}{2}-1}(\kappa(1+R)) 
& \forall \,
n \mbox{ even}\,\\[2mm]
J_{\frac{n}{2}-1}(\kappa(1+R))J_{-\frac{n}{2}+1}(\kappa R)
\,-\,J_{\frac{n}{2}-1}(\kappa R)J_{-\frac{n}{2}+1}(\kappa(1+R)) & \forall \,
n \mbox{ odd}\,,
\end{array}
\right.
\end{align}
where we note, that the equations in \eqref{Lapl_Allgem_EW_R} at
$n_L\,:=\,n+2$ are the equations for  
$\kappa\,=\kappa_{1,S,(n)}({\cal A})$ as well. 
For $n=3$ in the proof of \cite[Corollary 6]{RuRuTh2025} we showed the asymptotic limit
of equation \eqref{Lapl_Allgem_EW_R} as the following equation in
$\kappa$ 
\begin{align}\label{kappa_0}
 0\,=\,  \,\frac{2}{\pi \kappa} 
\sin(\kappa) \,\,,
\end{align} 
with the smallest positive solution $\kappa\,=\,{\pi}$.
In addition we will show another way from equations
\eqref{Lapl_Allgem_EW_R} to \eqref{kappa_0} in Note
\ref{limit_Asympt_EWPi} below 
which uses asymptotic formulas. 
But here, we use the ideas of Subsection \ref{Sec2avBou}
for the justification of the smallest positive solution  
$\kappa\,=\,{\pi}$ as the constant in the Friedrichs' inequality. 
Namely, one applies the result of \cite{nazarov2000}, Prop. 1.1 with
$p=q=2$ and $\kappa\,=\,\lambda^{(r)}_{p,q}(R,R+1)$. % (cf.\cite{nazarov2000}).
We write the last formula in the proof of the Propositon 1.1 and find
\begin{align}
\label{eq:poincareRinfty}
\hspace*{-1.62cm}
\kappa \,=\, \sqrt{
\min_w {\frac{\int_{{\Omega}_{(n),2R}} ({\underline{\nabla}} w)^{T} \cdot {\underline{\nabla}} w \, d\underline{x}}
{\int_{{\Omega}_{(n),2R}} \left|w\right|^2 \, d\underline{x}}}}\,=\,\sqrt{
\min_{\tilde{w}} 
\frac{\int_{{\Omega}_{(n),2R}} \!\!  r^{n-1} ({\underline{\nabla}} {\tilde{w}})^{T} \cdot {\underline{\nabla}} {\tilde{w}} \,d \omega_{(n)} d {r}}{\int_{{\Omega}_{(n),2R}} \!\!  r^{n-1} 
\left|\tilde w\right|^2 \, d \omega_{(n)} d {r} }}
\,. 
\end{align}
Finally, integral estimates lead to 
\begin{align}
\label{eq:poincareR_inequ}
\hspace*{-1.62cm}
\left(
{\frac{R}{R+1}}\right)^{\frac{n-1}{2}}{\pi}\,\leq\,
\kappa \,\leq\, \left(
{\frac{R+1}{R}}\right)^{\frac{n-1}{2}}{\pi}\, \,,
\end{align}
where we have applied Eq. \eqref{eq:1DEigenwert}. 
Passing to the limit we find $ \lim_{{R}\to\infty} {\kappa}_{1,L,(n)}({2R})\,=\,{\pi}$. 
Summing up we see with \eqref{eq:poincareR_inequ} that \\[.15cm]
$\lim_{{R}\to\infty} {\lambda}_{1,L,(n)}(2R)\,=\,
\lim_{{R}\to\infty} \kappa^2_{1,L,(n)}({2R})\,=\, $
$ \lim_{{\cal A}\to\infty} {\lambda}_{1,L,(n)}({\cal A})\,=\,
\lim_{{\cal A}\to\infty} \kappa^2_{1,L,(n)}({\cal A})\,=\,{\pi}^2$.}}
\end{proof}
\noindent Using Proposition
\ref{PiQuadratinf} we can understand the asymptotic behaviour
of $c_{p}({\cal A})$ and $c_{p,S}({\cal A})$ (cf. \eqref{simple-tool}) as
\begin{align*}
\lim_{{\cal A}\,\to\,\infty}{\;\!} c_{p,(n)}({\cal A})={\frac{1}{\pi}}\,\,\quad \text{and} \quad \lim_{{\cal A}\,\to\,\infty}{\;\!} c_{p,S,(n)}({\cal A})={\frac{1}{\pi}}.
\end{align*}
\begin{note} \label{limit_Asympt_EWPi} We can also  %describe the method to
                                %deduce from the equations
                                %\eqref{Lapl_Allgem_EW_R} the limit
                                %result \eqref{kappa_0} by means of
                                %the
  use asymptotic formulas. Namely,
we use the asymptotic behaviour of the {{Besselfunctions }}
\cite[(5.16.1) and (5.16.2) on p.~177]{Lebedev} as a tool in the next step,
where we write these formulas applying  %$\cal O$ -
'the Big O notation' which is also called Landau symbol.
The asymptotic formulas for the Bessel functions at $t\,\to\,\infty$ are 
\begin{align}\label{Asympt_EW_R}
\sqrt{\frac{2}{\pi \cdot t}}\left(\cos(t - {\frac{n-1}{4}}\pi)\,+\,{\cal O}(|t|^{-1})\right)\quad\, = \quad\,&\quad
J_{\frac{n}{2}-1}(t)\,\quad \quad\forall \,
n\,\in\,{\bb N}\,\,n>1\nonumber\\
\sqrt{\frac{2}{\pi \cdot t}}\left(\sin(t - {\frac{n-1}{4}}\pi) \,+\,{\cal O}(|t|^{-1})\right)\quad\, = \quad\,&
\left\{
\begin{array}{ll}
Y_{\frac{n}{2}-1}(t)\,& \forall \,
n \mbox{ even}\,\\[2mm]
J_{-\frac{n}{2}+1}(t)\, & \forall \,
n \mbox{ odd}\,.
\end{array}
\right.%\nonumber
\end{align}
We write Eqs. \eqref{Lapl_Allgem_EW_R} for large values of $R\,=\,\frac{\cal A}{2}$	
by means of the asymptotic formulas \eqref{Asympt_EW_R}. 
With some calculation using the elementary trigonometric formulas and
the well 
known angle sum identities we see for all $n\,\in\,{\bb N}\,;\,n>1$ that 
\begin{align}\label{Allgem_EW_R_Asympt}
\hspace*{-.3cm}
0=&\frac{2}{\pi \kappa {\sqrt{R(R+1)}}}\left( \cos\big(\kappa(R+1) - {\frac{n-1}{4}}\pi\big)\sin\big(\kappa R - {\frac{n-1}{4}}\pi\big)\,
-\, \cos\big(\kappa R - {\frac{n-1}{4}}\pi\big)
\sin\big(\kappa(R+1) - {\frac{n-1}{4}}\pi\big)
\right)\nonumber\\
{~} & +
\frac{2}{\pi \kappa {\sqrt{R(R+1)}}}\left({\cal O}(|\kappa R|^{-1}) \,+\, {\cal O}(|\kappa (R+1)|^{-1}) \,+\,
{\cal O}(|{\kappa}^2 R(R+1)|^{-1}) 
 \right) \nonumber\\
{~}\,=\, &
\frac{2}{\pi\kappa {\sqrt{R(R+1)}}}\left(-\,\sin(\kappa) \,+\,{\cal O}(|\kappa R|^{-1}) \,+\, {\cal O}(|\kappa (R+1)|^{-1}) \,+\,
{\cal O}(|{\kappa}^2 R(R+1)|^{-1}) 
 \right)\,  \,
\end{align}
as an 'approximate' limit of the Equations \eqref{Lapl_Allgem_EW_R} for large values of $R\,=\,\frac{\cal A}{2}$.
\end{note}
%%%
\begin{note} \label{limit_Laplace_EF}
The 'approximate' limit of the first eigenfunction $w_{1,L,{\cal A}}({\underline{x}})$ (cf. Theorem \ref{thm_Lapl_n} 
Eq. \eqref{Lapl_Allgem_EF}) is now immediately obtained\\[.1cm]
for large values of the number $R\,=\,\frac{\cal A}{2}$ (large ${\cal A}$) and $r\,:=\,\frac{\cal A}{2}+s\,=\,R+s \,$ at $s\,\in\,[0,1]$
as
\begin{align*}
w_{1,L,{\cal A}}({\underline{x}})\,\approx\,{\tilde{c}}\,{\sin}(\pi s)
  \,\,.
\end{align*}	
In order to prove the 'approximate' limit one has to
expand the fractions in $w_{1,L,{\cal A}}({\underline{x}})$ to use 
the relations
\eqref{Asympt_EW_R} and
\eqref{Bess_Summ} for $t\,\to\,\infty$. The explicit proofs for the 'approximate' limits 
at $n=2$ and $n=3$ for the first Laplace eigenfunctions and one of the first Stokes eigenfunctions are 
given in \cite{RuRuTh2016} and \cite{RuRuTh2025} , where it was possible to ensure periodic conditions 
with respect to one argument $t \,\sim \,R\cdot \varphi \,$ as the argument of an periodic strip in the 
choice of function spaces in \cite{RuRuTh2016} too.
Unfortunately this 
procedure is restricted to the case $n=2$. So we maintain the arguments
${\vartheta_{1}}, {\vartheta_{2}},\dots   
, {\vartheta_{n-2}}, {\varphi}\,$ for  $n\geq 3$. 
Additionally we can specify an `approximate' limit for the first
Stokes eigenfunctions ${\underline{w}}_{{\;\!}\varphi,S,{\cal A}}({\underline{x}})$ 
(cf. Eq. \eqref{Stokes_Allgem_EF}) for
large ${\cal A}$ as $s\,\in\,[0,1]$
\begin{align*}
{\underline{w}}_{{\;\!}\varphi,S,{\cal A}}({\underline{x}})\,\approx\,{\tilde{c}_{{\;\!}}}
\left\{
\begin{array}{ll}
\, {\sin}(\pi s)  \sin{\vartheta_{1}}\sin{\vartheta_{2}}\cdots \sin{\vartheta_{n-2}}\cdot 
{\underline{\mathfrak{e}}}_{{\;\!}\varphi}\,& \forall \,n\,\in\,{\bb N}\,;\,n>2  \\
\,  {\sin}(\pi s) \cdot 
{\underline{\mathfrak{e}}}_{{\;\!}\varphi}\,& 
\text{for }\,
\,n=2 \, .
\end{array}\right.
\end{align*}
\end{note} 
%%%%%%%%%%%%%%
%%%%%%%%%%%%%%##
\section{Results and Calculations}\label{Sec4}
To the best knowledge of the authors the results in Theorems
\ref{thm_Lapl_n} and \ref{thm_Stokes_n} are completely new.
For the first time we have constructed
\begin{itemize}
  \item the first Laplace
eigenfunctions and first Laplace eigenvalues for annuli 
in ${\bb R}^n,\,n\,>\,3$,
\item the first Stokes eigenfunctions and the first Stokes eigenvalues for balls and annuli
  in ${\bb R}^n,\,n\,>\,3$,
\item the connection between the 
first Laplace eigenvalues and the the first Stokes eigenvalues in the shift Theorem \ref{thm_shift_Lapl_Stokes_n}.
\end{itemize}
% There one has to take into account, that only their inverse values $\kappa_{1,L,(n)}({\cal A})$ and $\kappa_{1,S,(n)}({\cal A})$
% are computable. We have there only to calculate the $c_{p,(n)}({\cal A})$ respectively the $\kappa_{1,L,(n)}({\cal A})$
% because of the resuts from the shift Theorem \ref{thm_shift_Lapl_Stokes_n}.\\
The last result impacts the numerical calculation of the Poincar\'e
constants $c_{p,(n)}({\cal A})$ and $c_{p,S,(n)}({\cal A})$ since they
fulfill $c_{p,S,(n)}({\cal
  A})\,=\,c_{p,(n+2)}({\cal A})$ and  
the related roots are connected through 
$\kappa_{1,S,(n)}({\cal A})\,=\,\kappa_{1,L,(n+2)}({\cal A})$
$ \forall\,{\cal A}\,\in\,[0,\infty)$ and  $\forall\,n \,\in\,{\bb N}
,\,n >1$.
% There are the 
% $\kappa_{1,S,(n)}({\cal A})\,=\,\kappa_{1,L,(n+2)}({\cal A})$ (cf. the Theorems \ref{thm_Lapl_n} and \ref{thm_Stokes_n})
% defined as roots of the corresponding transcendental equations.
So we only calculate the '$\kappa_{1,L,(n+2)}({\cal A})$',
since the transfer to the $\kappa_{1,S,(n)}({\cal A})$ is obvious.
We note, that ${\cal A}$ has shown to be an excellent scaling parameter
for the calculation of the $\kappa_{1,L,(n)}({\cal A})$ $\forall\,n \,\in\,{\bb N} ,\,n >1$
over the range of ${\cal A}\,\in\,[0,\infty)$.  
The big advantage is a very small range for the roots $\kappa_{1,L,(n)}({\cal A})$ of the
transcendental equations.

\begin{figure}[th]
\centering
\scalebox{1.2}[1.6]{\includegraphics[width=0.55\textwidth]{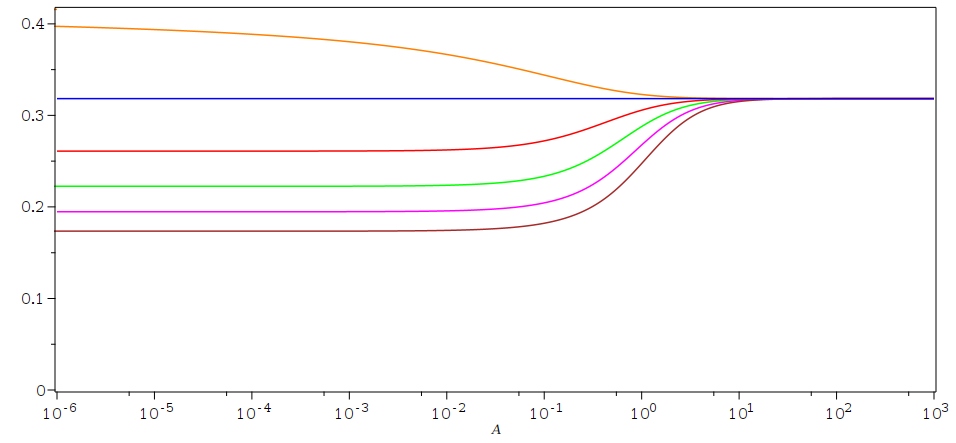}}
\caption{Calculated (Laplace) Poincar\'{e} constants
  % {$c_{p}({\calA})$
  for $n\,=\,{\Orange{2}},{\blue{3}},{\red{4}},{\green{5}},{\magenta{6}},{\brown{7}}$
as functions of $\mathcal A$ in logarithmic 
scale}
\label{fig:Poincare_N2bis_plot}
\end{figure}

We use standard Maple-procedures for the solution
$\kappa_{L,(n)}({\cal A})$ of transcendental equations  
\eqref{Lapl_Allgem_EW}.
The results of our calculations were written in data files as
${\cal A}$-pointwise exact $\kappa_{1,L,(n)}({\cal A})$-values. 
The ${\cal A}$-grid was chosen very fine for small ${\cal A}$-values and rougher for large ${\cal A}$
because of the completely different behaviour near ${\cal A}=0$ and
${{\cal A}\,\to\,\infty}$, respectively.  
The calculations of the $c_{p,(n)}({\cal A})$ are executed with
pointwise Maple procedures. 
The graphical representations of the obtained results in Figures 2 and
3 below and Figures 4 and 5 in the Appendix 
are created with Maple as well. 
We use a logarithmic scale for ${\cal A}$.
The values of $\kappa_{1,L,(n)}({\cal A})$ at ${\cal A}\,=\,0$ are the
first roots of the corresponding Besselfunctions. 
In ascending dimensional order
$n\,=\,{\Orange{2}},{\blue{3}},{\red{4}},{\green{5}},{\magenta{6}},{\brown{7}}$
 %$n\,= \,2, \,3 ,\,4 ,\,5,\,6,\,7\dots$\,:
  we find $\kappa_{1,L,(2)}({o})\,\approx\, 2.404825558 $,\,\, 
$\kappa_{1,L,(3)}({o})\,=\,\pi $,\,\,$\kappa_{1,L,(4)} ({o})\,\approx\, 3.831705970 $,\,\,
$\kappa_{1,L,(5)}({o})\,\approx\,4.493409458 $,\,\,$\kappa_{1,L,(6)}({o})\,\approx\, 5.135622302 $,
$\kappa_{1,L,(7)}({o})\,\approx\,5.763459197 $,\,\,$\dots$\,\,.%\\[.2cm]
\begin{figure}[h]
\centering
\scalebox{1.2}[1.6]{\includegraphics[width=0.55\textwidth]{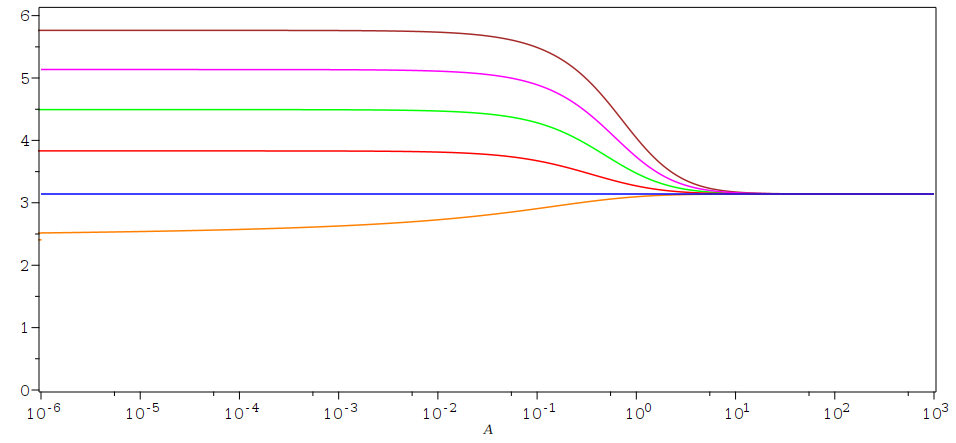}}
\caption{Calculated roots $\kappa_{1,L}({\cal A})$ %of the Laplace
  % transcendental equations
  as functions of $\mathcal A$
  in logarithmic 
scale $n\,=\,{\Orange{2}},{\blue{3}},{\red{4}},{\green{5}},{\magenta{6}},{\brown{7}}$}
\label{fig:EWWurzel_N2bis_plot}
\end{figure}

Finally, let us state some results for the Laplace and
Stokes Poincar\'{e} constants (i.e. the first Laplace and Stokes
eigenvalues, respectively)
for the small-gap limit. %${\cal A} \to \infty$.
The limits of
{$c_{p,(n)}({\cal A})$} and {$c_{p,S,(n)}({\cal A})$} 
for ${\cal A} \to \infty$ are equal and 
\begin{align*}
\lim_{{\cal A}\,\to\,\infty}{\;\!} c_{p,(n)}({\cal A}) \,=\,\lim_{{\cal A}\,\to\,\infty}{\;\!} c_{p,S,(n)}({\cal A})={\frac{1}{\pi}}\,\,\quad\quad \forall\,n \,\in\,{\bb N} ,\,n >1\,\,.
\end{align*}
These limits are independent of the dimension of the annuli. 
Moreover, for the first
eigenvalues $\lambda_{1,L,(n)}({\cal A})$ and 
$\lambda_{1,S,(n)}({\cal A})$ for any ${\cal A} \in (0,\infty)$ we see, that 
\begin{align*}
\lim_{{\cal A}\,\to\,\infty}{\;\!} \lambda_{1,L,(n)}({\cal A}) \,=\,{\pi}^2 \quad \text{or}\quad
\lim_{{\cal A}\,\to\,\infty}{\;\!} \lambda_{1,S,(n)}({\cal A}) \,=\,{\pi}^2 
\,\,.
\end{align*}
The multiplicity of the first eigenvalues $\lambda_{1,L,(n)}({\cal A})$ is one
and multiplicity of the first eigenvalues is $\lambda_{1,S,(n)}({\cal A})$ is constant. $\lambda_{1,L,(n)}({\cal A})$ and
$\lambda_{1,S,(n)}({\cal A})$ are continuous function of
${\cal A} \in (0,\infty)$ too. 
Finally, we have that 
${\lambda}_{1,L,(n)}({\cal A})\,\leq\,{\lambda}_{1,S,(n)}({\cal A})$
independenly of the dimension $n$ of the annuli. Thus, $c_{p,S,(n)}({\cal A})\,\leq
\,c_{p,(n)}({\cal  A}) $ for all  ${\cal A} \in (0,\infty)$ and for all $n \,\in\,{\bb N} ,\,n >1$. \\
Additionally our investigations show, that we have a universal upper
bound for Poincar\'{e} constants for  
all  ${\cal A} \in [0,\infty)$:
\begin{align*} 
c_{p,S,(n)}({\cal A}) \,\leq\, \frac{1}{\pi} \quad \, \forall \,n \,\in\,{\bb N} ,\,n >1
& \quad \quad \text{ and }\quad 
c_{p,(n)}({\cal A}) \,\leq\, \frac{1}{\pi}\,\,\quad \forall \,n \,\in\,{\bb N} ,\,n \geq  3\,\,.
\end{align*}
%%%%
%%%%%%

%%%%%%%%%%%%%%%%%%%%%%%%%%%%
\section*{Appendix}\label{App_Bes}
Let (cf. Notation \ref{N3}) the unit vectors in the Cartesian coordinate system in ${\bb
  R}^{n},\,n\geq 3$
be given by ${\underline{\mathfrak{e}}}_{j}\,:=\,
(\delta_{j,1},\delta_{j,2}, \dots, \delta_{j,n})^{T}$ for all $j=1,2,\dots,n$, with
Kronecker's delta $\delta_{j,k}$. 
The polar coordinates are stated as $r$,  $\vartheta_{1}$, $\dots$ , $\vartheta_{n-2}$ and $\varphi$ 
(cf. the following Definition \ref{polarcoord}) with the corresponding unit vectors
${\underline{\mathfrak{e}}}_{{\;\!}r}$, ${\underline{\mathfrak{e}}}_{{\;\!}\vartheta_{1}}$, $\dots$, ${\underline{\mathfrak{e}}}_{{\;\!}\vartheta_{n-2}}$  and
${\underline{\mathfrak{e}}}_{{\;\!}\varphi}$.
\begin{definition} \label{polarcoord}
The representation of any point ${\underline{x}} \in {\bb R}^{n}$ in the system
of polar coordinates is given via
\begin{align}  \label{Pol1Sh}
\begin{array}{rcl}
x_1 &\,=\,&r\cdot \sin{\vartheta_{1}}\cdots\sin{\vartheta_{n-2}}\cos{\varphi}\\
x_2 &\,=\,&r\cdot \sin{\vartheta_{1}}\cdots\sin{\vartheta_{n-2}}\sin{\varphi}\\
x_3 &\,=\,&r\cdot \sin{\vartheta_{1}}\cdots\cos{\vartheta_{n-2}}\\
{~} & \vdots  & {~} \\
x_{n-1} &\,=\,&r\cdot \sin{\vartheta_{1}}\cos{\vartheta_{2}}\\
x_{n} &\,=\,&r\cdot \cos{\vartheta_{1}}
\end{array}\\
\mbox{with}\,\,
r\,:=\,\|{\underline{x}}\| \in [0,\infty) ,\,\vartheta_{1},\dots,\vartheta_{n-2}\,\in\,[0,\pi]
\,\,\mbox{and}\,\,\varphi\,\in\,[0,2\pi] \,\,.\nonumber
\end{align}
\end{definition}
\begin{notation}[Surface harmonics of degrees $\ell = 0$ and $\ell = 1$]\label{SurfHarm0+1}
The function $f(\vartheta_{n-2},\dots\,\vartheta_{1},\varphi)\,=\,const.\,\not=\,0$  is a non-vanishing harmonic 
polynomial of the degree $\ell = 0$ in $r$. We state for any $ n\,\in {\bb N}:n\,>\,1$ the sherical 
surface harmonic function
\begin{align}  \label{SphSurf0}
S^{\{0\}}\,:=\,1
\end{align}
In the representation \eqref{Pol1Sh} are the functions $\{x_{k}\}_{k=1}^n$ harmonic polynomials
of the degree $\ell = 1$ in $r$. We call the functions
 \begin{align}  \label{SphSurf1}
\begin{array}{rcl}
S^{\{1\}}_1 &\,=\,& \sin{\vartheta_{1}}\cdots\sin{\vartheta_{n-2}}\cos{\varphi}\\
S^{\{1\}}_2 &\,=\,& \sin{\vartheta_{1}}\cdots\sin{\vartheta_{n-2}}\sin{\varphi}\\
S^{\{1\}}_3 &\,=\,& \sin{\vartheta_{1}}\cdots\cos{\vartheta_{n-2}}\\
{~} & \vdots  & {~} \\
S^{\{1\}}_{n-1} &\,=\,& \sin{\vartheta_{1}}\cos{\vartheta_{2}}\\
S^{\{1\}}_{n} &\,=\,& \cos{\vartheta_{1}}\,,
\end{array}\\
\mbox{where}\,\,
\,\vartheta_{1},\dots,\vartheta_{n-2}\,\in\,[0,\pi]
\,\,\mbox{and}\,\,\varphi\,\in\,[0,2\pi] \,\,,\nonumber
\end{align}
sherical surface harmonics of the degree $\ell = 1$. We write $S\,\in\,{\mbox{span}}\{S^{\{1\}}_k\}_{k=1}^n$ 
for a spherical surface harmonic function of the degree $\ell = 1$
(e.g. the Definition in Subsection  6.3.1 in \cite{Triebel}) too .
\end{notation}
\begin{remark}\label{JacobiMat}
The first step in the calculation of
the transformation between the Cartesian coordinates and the spherical polar coordinates
(cf. Definition \ref{polarcoord})
is it to calculate the Jacobian Matrix:
\begin{align*}
{\underline{\underline{J}}} =
\begin{bmatrix} 
\frac{\partial x_1}{\partial r}&\frac{\partial x_1}{\partial \vartheta_{1}}&
\dots &\frac{\partial x_1}{\partial \vartheta_{n-2}}
&\frac{\partial x_1}{\partial{\varphi}}\\
\vdots&\vdots & {~}& \vdots&\vdots\\
\frac{\partial x_n}{\partial r}&\frac{\partial x_n}{\partial \vartheta_{1}}&
\dots  &\frac{\partial x_n}{\partial \vartheta_{n-2}}&\frac{\partial x_n}{\partial{\varphi}}
\end{bmatrix}
  \end{align*}
 and in a second step the corresponding metric tensor 
 ${\underline{\underline{g}}}$:
\begin{align*}  
{\underline{\underline{g}}}\,:=\,{\underline{\underline{J}}} \cdot {\underline{\underline{J}}}^T\,=\,
\text{diag}\{g_{j,j}\}_{j=1}^{n}\quad\,\text{with}\,\, g_{1,1}=1,\,g_{2,2}=r^2,\,\dots\,,
g_{n,n}=r^2{\sin}^2{\vartheta_{1}}  \cdots {\sin}^2{\vartheta_{n-2}}
\,.
\end{align*}
The reciprocal of the square roots of the $\{g_{j,j}\}_{j=1}^{n}$ used as multipliers column by column provide
applied on ${\underline{\underline{J}}}$ the matrices of transformation.
\end{remark}
\begin{remark}
\label{R1} The transformation between the Cartesian coordinates and the spherical polar coordinates
(cf. Notation \ref{N3}) as the transformation of one coordinate system to the other 
can be written as ${\underline{u}}_{\mathfrak{c}}\,=\,
{\underline{\underline{T}}}_{{\mathfrak{c}},{\mathfrak{s}}}{\underline{u}}_{\mathfrak{s}}$
\,or\,\,\,${\underline{u}}_{\mathfrak{s}}\,=\,
{\underline{\underline{T}}}^{-1}_{{\mathfrak{c}},{\mathfrak{s}}}
{\underline{u}}_{\mathfrak{c}}\,=\,
{\underline{\underline{T}}}_{{\mathfrak{s}},{\mathfrak{c}}}{\underline{u}}_{\mathfrak{c}}$,
respectively,
where we  use the concept of columns of coordinates
and the transformation matrices ${\underline{\underline{T}}}_{{\mathfrak{s}},{\mathfrak{c}}}\,:=\,
{\underline{\underline{T}}}_{{\mathfrak{c}},{\mathfrak{s}}}^{-1}\,=\,
{\underline{\underline{T}}}_{{\mathfrak{c}},{\mathfrak{s}}}^{T}\,$ 
\begin{equation} \label{R1_T=2} 
		{\underline{\underline{T}}}_{{\mathfrak{c}},{\mathfrak{s}}}\,
		:=\,\left[
		\begin{array}{lc}
\cos{\varphi}& -\sin{\varphi} \\
\sin{\varphi}&  {~}\cos{\varphi}
		\end{array}
		\right] \quad\,\mbox{for} \quad \,
n\,=\,2\,\,,
\end{equation}
\begin{equation}\label{R1_T=3} 
		{\underline{\underline{T}}}_{{\mathfrak{c}},{\mathfrak{s}}}\,
		:=\,\left[
		\begin{array}{llc}
\sin{\vartheta_{1}}\cos{\varphi}& \cos{\vartheta_{1}}\cos{\varphi}& -\sin{\varphi} \\
\sin{\vartheta_{1}}\sin{\varphi}& \cos{\vartheta_{1}}\sin{\varphi}& {~}\cos{\varphi}\\	
\cos{\vartheta_{1}} & -\sin{\vartheta_{1}}& {~} 0
		\end{array}
		\right] \quad\,\mbox{for} \quad \,
n\,=\,3\,\,,
\end{equation}
\begin{equation}\label{R1_T=4} 
		{\underline{\underline{T}}}_{{\mathfrak{c}},{\mathfrak{s}}}\,
		:=\,\left[
		\begin{array}{llcc}
\sin{\vartheta_{1}}\sin{\vartheta_{2}}\cos{\varphi} & \cos{\vartheta_{1}}\sin{\vartheta_{2}}\cos{\varphi}& \cos{\vartheta_{2}}\cos{\varphi}&
-\sin{\varphi} \\
\sin{\vartheta_{1}}\sin{\vartheta_{2}}\sin{\varphi} & \cos{\vartheta_{1}}\sin{\vartheta_{2}}\sin{\varphi}& \cos{\vartheta_{2}}\sin{\varphi} & {~}\cos{\varphi}\\
\sin{\vartheta_{1}}\cos{\vartheta_{2}} & \cos{\vartheta_{1}}\cos{\vartheta_{2}} & -\sin{\vartheta_{2}}& {~} 0\\
\cos{\vartheta_{1}} & -\sin{\vartheta_{1}}& {~} 0 & {~} 0
		\end{array}
		\right] \quad\,\mbox{for} \quad \,
n\,=\,4\,\,,
\end{equation}
and for general $n=n$
\begin{equation}\label{R1_T=n} 
		{\underline{\underline{T}}}_{{\mathfrak{c}},{\mathfrak{s}}}\,
		:=\,\left[
\begin{array}{llccc}
\sin{\vartheta_{1}}\sin{\vartheta_{2}}\cdots \sin{\vartheta_{n-2}}\cos{\varphi} & 
\cos{\vartheta_{1}}\sin{\vartheta_{2}}\cdots \sin{\vartheta_{n-2}}\cos{\varphi}& \dots &
\cos{\vartheta_{n-2}}\cos{\varphi}&
-\sin{\varphi} \\
\sin{\vartheta_{1}}\sin{\vartheta_{2}}\cdots \sin{\vartheta_{n-2}}\sin{\varphi} & 
\cos{\vartheta_{1}}\sin{\vartheta_{2}}\cdots \sin{\vartheta_{n-2}}\sin{\varphi} & \dots &\cos{\vartheta_{n-2}}\sin{\varphi}&
 {~} \cos{\varphi}\\
\sin{\vartheta_{1}}\sin{\vartheta_{2}}\cdots \cos{\vartheta_{n-2}}& 
\cos{\vartheta_{1}}\sin{\vartheta_{2}}\cdots \cos{\vartheta_{n-2}} & \dots & - \sin{\vartheta_{n-2}}&
 {~} 0\\
\vdots & \vdots & \ddots &\vdots & \vdots\\
\sin{\vartheta_{1}}\cos{\vartheta_{2}} & \cos{\vartheta_{1}}\cos{\vartheta_{2}} &\dots &
 {~} 0 & {~} 0\\
\cos{\vartheta_{1}} & -\sin{\vartheta_{1}}& \dots &    {~} 0 & {~} 0
		\end{array}
		\right] %\,\,\mbox{for}  \,\, n\,=\,n\,
                \,.
\end{equation}
\end{remark}
%%%%%%%%%%%%
\begin{remark}\label{Lapldimn}
The Laplacian in spherical coordinates
$\Delta_{sph}(.):=\Delta_{r,\vartheta_{1},\dots\,\vartheta_{n-2},\varphi}(.)$ is
($n\geq3$; see, e.g.,  6.3.4 in \cite{Triebel})
\begin{align*} \hspace*{-.3cm}
 \Delta_{sph}(.)  = \frac{1}{r^{n-1}}
\frac{\partial{ }}{\partial r}({r^{n-1}}\frac{\partial { (.)}}{\partial r})
+\frac{1}{r^{2}}\left(
\frac{1}{\sin^{n-2}{\vartheta_{1}}}\frac{\partial{ }}{\partial \vartheta_{1}}
(\sin^{n-2}\vartheta_{1}\frac{\partial{ }(.)}{\partial \vartheta_{1}})+
\frac{1}{\sin^{2}\vartheta_{1}\sin^{n-3}{\vartheta_{2}}}\frac{\partial{ }}{\partial \vartheta_{2}}
(\sin^{n-3}\vartheta_{2}\frac{\partial{ }(.)}{\partial \vartheta_{2}})\,+\right.
\\
{~} 
\left.\,\dots\,+\,
\frac{1}{\sin^{2}\vartheta_{1}\sin^{2}{\vartheta_{2}}\cdots
\sin^{2}{\vartheta_{n-3}}
\sin^{~}{\vartheta_{n-2}}
}\frac{\partial{ }}{\partial \vartheta_{n-2}}
(\sin \vartheta_{n-2}\frac{\partial{ }(.)}{\partial \vartheta_{n-2}})\,+
\,
\frac{1}{\sin^{2}\vartheta_{1}\sin^{2}{\vartheta_{2}}\cdots
\sin^{2}{\vartheta_{n-2}}}
  \frac{\partial^{2}{ (.)}}{\partial \varphi^{2}}\right)\,\,,
\end{align*}
resp. $\displaystyle
\Delta_{r,\vartheta_{1},\dots\,\vartheta_{n-2},\varphi}(.)  = 
{~} 
\,\frac{1}{r^{n-1}}
\frac{\partial{ }}{\partial r}({r^{n-1}}\frac{\partial { (.)}}{\partial r})
\,-\,\frac{\displaystyle{1}}{\displaystyle{r^{2}}}{ \mbox{B}}{\;\!}(.)\,,\hspace*{.2cm}{~}
$
where B$(.) $ denotes the Beltrami differential operator.
\end{remark}
\noindent We define the Laplace-Beltrami operator by means of
Beltrami's differential operator
in the following
\begin{definition}\label{DBeltr} { For all $Y\,\in\,D({\boldsymbol
  B^{\circledast}})=C^{\infty}({\omega_{(n)}}) \subset {\bb
  L}_{2}({\omega_{(n)}})\,$  the Laplace-Beltrami operator is defined
as}
\begin{align*}
{\boldsymbol B^{\circledast}}\,{Y} := \,{ \mbox{B}}{\;\!}(Y)\,\,=\,-
\left(
\frac{1}{\sin^{n-2}{\vartheta_{1}}}\frac{\partial{ }}{\partial \vartheta_{1}}
(\sin^{n-2}\vartheta_{1}\frac{\partial{ } Y}{\partial \vartheta_{1}})+
\frac{1}{\sin^{2}\vartheta_{1}\sin^{n-3}{\vartheta_{2}}}\frac{\partial{ }}{\partial \vartheta_{2}}
(\sin^{n-3}\vartheta_{2}\frac{\partial{ } Y }{\partial \vartheta_{2}})\,+\right.\hspace*{3.1cm}  
\\
{~} 
\left.\,\dots\,+\,
\frac{1}{\sin^{2}\vartheta_{1}\sin^{2}{\vartheta_{2}}\cdots
\sin^{2}{\vartheta_{n-3}}
\sin^{~}{\vartheta_{n-2}}
}\frac{\partial{ }}{\partial \vartheta_{n-2}}
(\sin \vartheta_{n-2}\frac{\partial{ } Y}{\partial \vartheta_{n-2}})\,+
\,
\frac{1}{\sin^{2}\vartheta_{1}\sin^{2}{\vartheta_{2}}\cdots
\sin^{2}{\vartheta_{n-2}}}
\frac{\partial^{2}{\, Y }}{\partial \varphi^{2}}\right)\,\,.
\end{align*}
We denote the Friedrichs' extension of ${\boldsymbol B^{\circledast}}$
by ${\boldsymbol 
  B}:={\overline{\boldsymbol B^{\circledast}}}$, where ${\boldsymbol
  B}$ is applied on 
$D({\boldsymbol B})\,:=\,{\bb W}_{2}^{2}({\omega_{(n)}})\subset {\bb L}_{2}({\omega_{(n)}})$.
\end{definition}
\begin{remark} The detailed construction of the Laplace-Beltrami operator ${\boldsymbol
  B}$ is given in \cite[Subsection 6.3.5]{Triebel} at great length. Especially the step from a n-dimensional 
shell to the boundary ${\omega_{(n)}}$ is illustrated there.
\end{remark}
\noindent We cite explicitly the following result:
\begin{theorem} \label{thmsbeltr}
The Laplace-Beltrami operator ${\boldsymbol B}$ is nonnegative and self-adjoint.
${\boldsymbol B}$ is an operator with pure point spectrum. Its eigenvalues are $\ell (\ell+n-2)$,  
$\ell\,=\,0,1,2,\dots $.
The surface harmonics S(.) of the degree $\ell$ form a set of all eigenfunctions of ${\boldsymbol B}$
to the eigenvalue $\ell (\ell+n-2)$.
\end{theorem}
\begin{remark} \label{Div_phi}
In our study of the first Stokes eigenfunctions we need the divergence of
vector function of the form  
${\underline{v}}:= v_{{\;\!}\varphi}(r,\vartheta_{1},\dots,\vartheta_{n-2},\varphi)\cdot
{\underline{\mathfrak{e}}}_{{\;\!}\varphi}$ in n-dimensional
(spherical) polar coordinates. We have 
\begin{equation} \label{DIV=2} 
		\mbox{div}{\;\!}\,{\underline{v}}=\mbox{div}{\;\!}\,v_{{\;\!}\varphi}(.)\cdot{\underline{\mathfrak{e}}}_{{\;\!}\varphi}\,=
		\,\frac{1}{r}\frac{\partial{\,v_{{\;\!}\varphi} }}{\partial \varphi}\,\quad\,\mbox{for} \quad \,
n\,=\,2\,\,,
\end{equation}
\begin{equation}\label{DIV=3} 
\mbox{div}{\;\!}\,{\underline{v}}=\mbox{div}{\;\!}\,v_{{\;\!}\varphi}(.)\cdot{\underline{\mathfrak{e}}}_{{\;\!}\varphi}\,=\,\frac{1}{r \sin\vartheta_{1}}\frac{\partial{\,v_{{\;\!}\varphi} }}{\partial \varphi}\,		\quad\,\mbox{for} \quad \,
n\,=\,3\,\,,
\end{equation}
\begin{equation}\label{DIV=4}
\mbox{div}{\;\!}\,{\underline{v}}=\mbox{div}{\;\!}\,v_{{\;\!}\varphi}(.)\cdot{\underline{\mathfrak{e}}}_{{\;\!}\varphi}\,=\,\frac{1}{r \sin\vartheta_{1}\sin{\vartheta_{2}}}\frac{\partial{\,v_{{\;\!}\varphi} }}{\partial \varphi}\,		
\quad\,\mbox{for} \quad \,
n\,=\,4\,\,,
\end{equation}
and $\quad \dots$
\begin{equation}\label{DIV=n}
\mbox{div}{\;\!}\,{\underline{v}}=\mbox{div}{\;\!}\,v_{{\;\!}\varphi}(.)\cdot{\underline{\mathfrak{e}}}_{{\;\!}\varphi}\,=\,\frac{1}{r \sin{\vartheta_{1}}\sin{\vartheta_{2}}\cdots \sin{\vartheta_{n-2}}}\frac{\partial{\,v_{{\;\!}\varphi} }}{\partial \varphi}\,\quad	
\,\,\mbox{for}  \,\,\quad
n\,=\,n\,\,,
\end{equation}
where the above statements follow by simple calculations.
\end{remark}
%%%
\noindent For the Besselfunctions $J_{\frac{n}{2}}(t)\,,J_{-\frac{n}{2}}(t) \,;\,n\,\in\,{\bb N}$  (and for the Weber functions 
$Y_{n}(t), \,n\,\in\,{\bb N}_{o}$)
we can use representations by trigometric functions respectively series in $t$
(cf. \cite{Triebel}, 5.5.1, \cite{CouHil} or  \cite{Lewin}) e.g.
\begin{eqnarray} \label{Bess_Halbe}
J_{\frac{1}{2}}(t)\,=\,\sqrt{\frac{2}{t\pi}}\cdot\,{\sin(t)} \quad  \quad ,\quad
J_{-\frac{1}{2}}(t)\,=\,\sqrt{\frac{2}{t\pi}}\cdot\,{\cos(t)} \, \mbox{ , }
\end{eqnarray}
\begin{eqnarray} \label{Bess_DreiHalbe}
J_{\frac{3}{2}}(t)\,=\,\sqrt{\frac{2}{t\pi}}\cdot\,\left(-
{\cos(t)}\,+\,
\frac{\sin(t)}{t} 
\right)
\quad \,=\,
\frac{J_{\frac{1}{2}}(t)}{t} \,-\,J_{-\frac{1}{2}}(t)
\quad\nonumber \mbox{ and }\\
{~}\\
J_{-\frac{3}{2}}(t)\,=\,-\sqrt{\frac{2}{t\pi}}\cdot\,\left(
{\sin(t)}\,+\,
\frac{\cos(t)}{t} 
\right)
\quad \,=\,-\left(
\frac{J_{-\frac{1}{2}}(t)}{t} \,+\,J_{\frac{1}{2}}(t)\right)\,.
 \nonumber 
\end{eqnarray}
%%%%
\noindent A tool in the proof of the asymptotic behaviour of the eigenfunctions for ${\cal A}\,\to\,\infty$ is e.g.
the representation of the Besselfunctions $J_{-\beta}(t) \,;\,\beta\,\in\,\,(-\infty,0)$
\begin{eqnarray} \label{Bess_Summ}
J_{-\beta}(t)\,=\,\sum_{j=0}^{\infty} (-1)^{j}
{\frac{1}{j!\Gamma(j+1-\beta)}}\left(\frac{t}{2}\right)^{2j-\beta} \, \mbox{ . }
\end{eqnarray}
\begin{figure}[ht]
\centering
\scalebox{1.2}[1.6]{\includegraphics[width=0.55\textwidth]{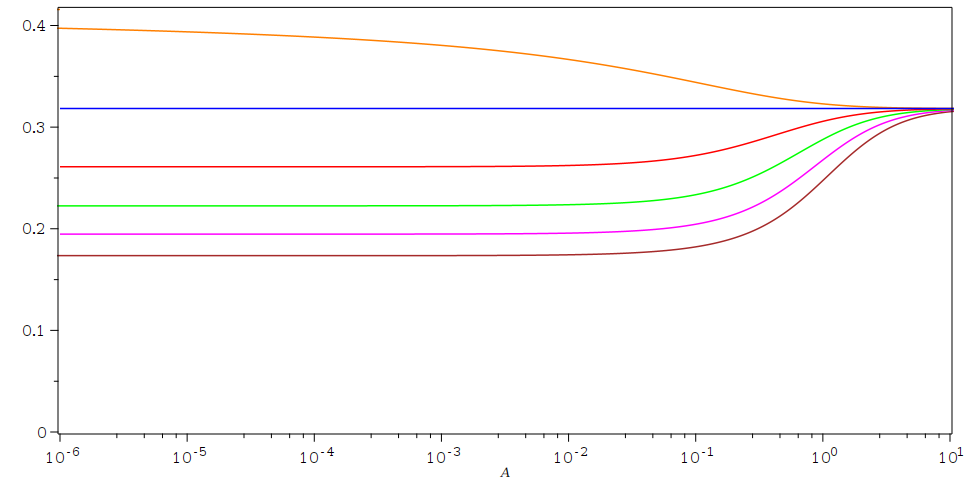}}
\caption{Zoomed in view of Figure 2}
\label{fig:Lupe_P_N2bis_plot}
\end{figure}
\begin{figure}[ht]
\centering
\scalebox{1.2}[1.6]{\includegraphics[width=0.55\textwidth]{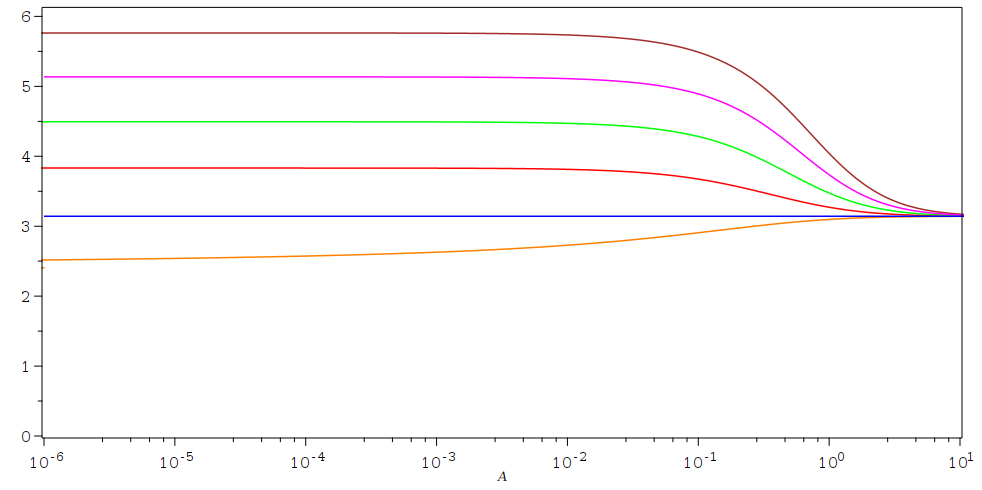}}
\caption{Zoomed in view of Figure 3}
\label{fig:Lupe_EWW_N2bis_plot}
\end{figure}
\end{document}

%% file: color_pdf.tex
\definecolor{dark-BLGR}{rgb}{0.28,0.28,0.28}

\definecolor{dark-red}{rgb}{0.90,0.14,0.14}

\definecolor{dark-viol}{rgb}{0.73,0.18,0.69}
\definecolor{Cyan}{rgb}{0.31,0.67,0.82}
\newcommand{\blue}{\textcolor{blue}}
\newcommand{\red}{\textcolor{red}}
\newcommand{\green}{\textcolor{green}}

\definecolor{light-yellow}{rgb}{1,1,0.8}

\definecolor{Blue}{rgb}{0.0353,0.0275,0.4}

\definecolor{GRAY}{rgb}{0.26,0.26,0.26}

\definecolor{GREYY}{rgb}{0.45,0.45,0.45}

\definecolor{vivid-viol}{rgb}{0.3255,0.0353,0.55}

\definecolor{dark-blue}{rgb}{0.05,0.05,0.65}

\definecolor{dark-green}{rgb}{0.03,0.77,0.29}

\definecolor{dark-Green}{rgb}{0.03,0.57,0.09}

\definecolor{strong-viol}{rgb}{0.2353,0.094,0.349}

\definecolor{BLUE}{rgb}{0.41,0.44,0.93}

\definecolor{RED}{rgb}{0.90,0.18,0.32}

\definecolor{Yel}{rgb}{0.89,0.95,0.19}

\definecolor{White}{rgb}{1,1,1}
\definecolor{black}{rgb}{0,0,0}

\definecolor{Orchid}{rgb}{0.6,0.1607,0.2}

\definecolor{Orange}{rgb}{0.99,0.521,0.34902}
\newcommand{\Orange}{\textcolor{Orange}}
\definecolor{magenta}{rgb}{0.99,0.02,0.99}
\newcommand{\magenta}{\textcolor{magenta}}
\definecolor{Thistle}{rgb}{0.1151,0.1249,0.9873}

\definecolor{brown}{rgb}{0.5450,0.2825,0.07453}
\newcommand{\brown}{\textcolor{brown}}

%% file: Poincare26.bbl
\begin{thebibliography}{99}
\bibitem{ADN} S.~Agmon, A.~Douglis, and L.~Nirenberg,
Estimates near the boundary for solutions of elliptic 
partial differential equations satisfying general boundary conditions II,
Comm. Pure Appl. Math. {\bf 17}, 35-92  (1964).
%%%%%%%
\bibitem{AGBu}
T.~Akinaga , S.C.~Generalis , F.H.~Busse;
Tertiary and Quaternary States in the Taylor-Couette System
in Chaos, Solitons and Fractals
Volume 109, April 2018, Pages 107-117
https://doi.org/10.1016/j.chaos.2018.01.033
%%%%%
\bibitem{AAR} G.E. ~Andrews, R.~Askey, and R.~Roy,
Special Functions,
(Cambridge Univ.Press, Cambridge, New York, 1999).
%%%%%%
\bibitem{Catbri} L.~Cattabriga,
{{Su un problema al contorno relativo si 
sistema di equazione di Stokes}},  
Rend. Mat. Univ. Padova {\bf 31},. 308-340 (1961).
%%%%%%%
\bibitem{CoFoi} P.~Constantin and C.~Foias,
{{Navier-Stokes Equations}}, 
(Univ.of Chic.Press, Chicago, 1988).	
%%%%%%%
\bibitem{CouHil} R.~Courant and  D.~Hilbert, 
{{Methoden der Mathematischen Physik}}, Vol.I and Vol.II
3. Aufl. (Springer, Berlin, Heidelberg, New York, 1968).
%%%%%%%%
\bibitem{galdi1998}
{G.P.~Galdi}
An introduction to the mathematical theory of the Navier-Stokes
  equations, Vol.~1: Linearised steady problems
(Springer, New York, 1998).
%%%%%%%%
\bibitem{GilTru} D. Gilbarg, N.S. Trudinger, Elliptic Partial Differential Equations of Second
Order, Grundlehren der mathematischen Wissenschaften 224, Reprint of the 1998
ed., (Springer, Berlin, Heidelberg, New York, 2001).
\bibitem{GirRav} V.~Girault and P.-A.~Raviart, 
{{Finite element approximation of the Navier-Stokes equations}}, 
(Springer, Berlin, 1979).
%%%%%%%
\bibitem{Jos}  D.D.~Joseph, 
{{Stability of Fluid Motions}},
Vol.I, (Springer, Berlin, Heidelberg, New York, 1976).
%%%%%%%%%%%
\bibitem{Junk} M.~Junk
{{Numerische Untersuchung der Stabilität der
Str\"omung im weiten Kugelspalt}}, (Cuvillier Verlag, Göttingen, 2005).  
%%
\bibitem{KaiWahl}
R.~Kaiser, W.~von~Wahl, {{A New Functional for the Taylor-Couette Problem in the Small-Gap Limit}},
in Mathematical theory in fluid mechanics, Pitman Research Notes in Mathematics, Series 354, 
editors: G.P. Galdi, J. Malek, J. Necas, 114-134 (1996).
%%
\bibitem{LeeRu} 
D.~S.~Lee and B.~Rummler, The Eigenfunctions of the Stokes Operator in Special 
Domains III, ZAMM {\bf 82},(2002) 399–407.
%%
\bibitem{Lebedev} N.~N.~Lebedev,
{Spezielle Funktionen und ihre Anwendung}
(BI Wissenschaftsverlag, Mannheim, 1973).
%
\bibitem{Lewin} W.~I.~Lewin und J.~I.~Grosberg
{Differentialgleichungen der mathematischen Physik}
(Verlag Technik, Berlin, 1952).
%
\bibitem{MoonSpencer} P.~Moon and D.~E.~Spencer, Field Theory Handbook, Including Coordinate Systems, Differential Equations, and Their Solutions, 2nd ed. (Springer-Verlag, New York, 1988).
%%
\bibitem{nazarov2000} A.I.~Nazarov,
The one-dimensional character of an extremum point of the
  {F}riedrichs inequality in spherical and plane layers,
Journal of Mathematical Sciences {\bf 102}, 5 (2000), 4473-4486.
%%%%%%%
\bibitem{passerini2009}
A.~Passerini, M.~R\r{u}\v{z}i\v{c}ka, and G.~Th{\"a}ter,
Natural convection between two horizontal coaxial cylinders,
ZAMM {\bf 89}, 5 (2009) 399-413.
%%%%%%%
\bibitem{passerini2010}
A.~Passerini, C.~Ferrario, M.~R\r{u}\v{z}i\v{c}ka, and G.~Th{\"a}ter,
Theoretical results on steady convective flows between horizontal
  coaxial cylinders,
SIAM Journal on Applied Mathematics {\bf 71}, 2 (2011) 465-486.
%%%%%%%%
\bibitem{passerini2024}
A.~Passerini, B.~Rummler, M.~R\r{u}\v{z}i\v{c}ka, and G.~Th{\"a}ter,
Natural Convection in the Horizontal Annulus: Critical Rayleigh Number for the steady Problem,
ZAMM : Volume 105, Issue 3, March 2025, 
https://doi.org/10.1002/zamm.202300535
%%%%%%%%%%%%%%%%%%%%%%
\bibitem{RumZyl} B.~Rummler, {
The Eigenfunctions of the Stokes Operator in Special Domains I}, ZAMM {\bf 77}, 8 (1997) 619–627. 
\bibitem{RumHab} B.~Rummler, 
{{Zur L\"osung der instation\"aren inkompressiblen 
Navier-Stokesschen Gleichungen in speziellen Gebieten}}, 
Magdeburg: Habilitation (1999/2000).
%%%%%%%
\bibitem{RumKug1} B.~Rummler, {The Eigenfunctions of the Stokes Operator in 
the open Unit Ball and in the open spherical Annulus},
Proc. of the  8th. Asian Computational Fluid Dynamics Conference, 
Hong Kong, 10-14 January, 2010
%%%%%%%
\bibitem{RumTh2024} B.~Rummler and G.~Th{\"a}ter, 
{{The Stokes Eigenvalue Problem on balls and annuli in three dimensions: Solutions with Poloidal
and Toroidal Fields}}, https://doi.org/10.48550/arXiv.2408.06948 (2024) 1-18
%%%%%%%%%%%
\bibitem{RuRuTh2016}
 B.~Rummler,, M.~R\r{u}\v{z}i\v{c}ka, and G.~Th{\"a}ter,
{Exact Poincar\'e constants in two-dimensional annuli}, ZAMM {\bf 97}, 1 (2017)
110–122.
 %%%%%%%%%
\bibitem{RuRuTh2025}
B.~Rummler,, M.~R\r{u}\v{z}i\v{c}ka, and G.~Th{\"a}ter,
{Exact Poincar\'e constants in three-dimensional annuli},
Arxiv - Ithaca, NY : Cornell University . https://doi.org/10.48550/arXiv.2506.13891 (2025)  1-12
\bibitem{Temam} R.~Temam, {{Navier-Stokes equations, theory and
      numerical 
analysis}}, 3rd edit., (North Holland, Amsterdam, 1984).
%%%%%%%
\bibitem{Triebel} H.~Triebel, {{Higher Analysis}}, 
  (Barth, Leipzig Berlin Heidelberg Amsterdam:, 1992).
  %%%%%%%%%%
\bibitem{Weid}
  J.~Weidmann, Stetige Abhängigkeit der Eigenwerte und
  Eigenfunktionen elliptischer Differentialoperatoren vom
  Gebiet. MATHEMATICA SCANDINAVICA, 54 (1984)
  51–69.
  %https://doi.org/10.7146/math.scand.a-12040
  %%%%%%
\end{thebibliography}
